\documentclass[12pt]{amsart}
\usepackage{amsmath} 
\usepackage{amssymb}
\usepackage{amsfonts}
\usepackage{a4}

\usepackage{ndstye}

\newcommand{\bc}{\mathbf C}

\newcommand{\br}{\mathbf R}

\newcommand{\bn}{\mathbf N}
\newcommand{\Cal}{\mathcal}

\newcommand{\gs}{\gtrsim}

\newcommand{\hess}{\operatorname{Hess}}

\newcommand{\id}{\operatorname{Id}}

\newcommand{\im}{\operatorname{Im}}

\newcommand{\ls}{\lesssim}

\newcommand{\mn}[1]{\Vert#1\Vert}

\newcommand{\mnorm}[1]{\Vert#1\Vert}

\newcommand{\ol}{\overline}

\newcommand{\re}{\operatorname{Re}}

\newcommand{\restr}[1]{\big|_{#1}}

\newcommand{\set}[1]{\left\{\,#1\,\right\}}

\newcommand{\st}{\Sigma _2}
\newcommand{\supp}{\operatorname{\rm supp}}

\newcommand{\w}[1]{\langle #1\rangle }

\newcommand{\wf}{\operatorname{WF}}

\newcommand{\wt}{\widetilde}

\numberwithin{equation}{section}

\begin{document}

\baselineskip 16pt
\lineskip 2pt
\lineskiplimit 2pt

\title[Subprincipal]{Operators of subprincipal type} 
\author[NILS DENCKER]{{\textsc Nils Dencker}}
\address{Centre for Mathematical Sciences, University of Lund, Box 118,
SE-221 00 Lund, Sweden}
\email{dencker@maths.lth.se}



\maketitle

\thispagestyle{empty}

\section{Introduction}

We shall consider the solvability for a classical pseudodifferential
operator $P \in  
{\Psi}^m_{cl}(M)$ on a $C^\infty$ manifold $M$. This means that $P$
has an expansion $p_m + p_{m-1} + \dots$ where $p_k  \in S^{k}_{hom}$
is homogeneous of 
degree $k$, $\forall\, k$, and $p_m = {\sigma}(P)$ is the principal
symbol of the operator. A pseudodifferential operator is said to be 
of principal type if the 
Hamilton vector field $H_{p_m}$ of the principal symbol does not
have the radial direction ${\xi}\cdot\partial_{\xi}$ on $p_m^{-1}(0)$,
in particular $H_{p_m} \ne 0$. We shall consider the case when the principal symbol
vanishes of at least second order at an involutive manifold~$\st$, then~$P$
is not of principal type. 

$P$ is locally solvable at a compact set $K \subseteq M$
if the equation 
\begin{equation}\label{locsolv}
Pu = v 
\end{equation}
has a local solution $u \in \Cal D'(M)$ in a neighborhood of $K$
for any $v\in C^\infty(M)$
in a set of finite codimension.  
We can also define microlocal solvability of~$P$ at any compactly based cone
$K \subset T^*M$, see Definition~\ref{microsolv}.

For pseudodifferential operators of principal type, it is
known~\cite{de:nt} \cite{ho:nec}  that 
local solvability is equivalent to condition (${\Psi}$) 
on the principal symbol, which means that
\begin{multline}\label{psicond} \text{$\im ap_m$
    does not change sign from $-$ to $+$}\\ 
 \text{along the oriented
    bicharacteristics of $\re ap_m$}
\end{multline}
for any $0 \ne a \in C^\infty(T^*M)$. The oriented bicharacteristics are
the positive flow-out of the Hamilton vector field $H_{\re ap_m} \ne
0$ on which $\re ap_m =0$, these are also called
semibicharacteristics of~$p_m$.
Condition~\eqref{psicond} is invariant under multiplication of
~$p_m$ with non-vanishing factors, and symplectic changes of
variables, thus it is invariant under conjugation of ~$P$ with elliptic
Fourier integral operators. Observe that the sign changes
in~\eqref{psicond} are reversed
when taking adjoints, and that it suffices to check~\eqref{psicond}
for some $a \ne 0$ for which $H_{\re ap} \ne 0$
according to~\cite[Theorem~26.4.12]{ho:yellow}.

For operators which are not of principal type, the situation is
more complicated and the solvability may depend on the lower order terms. 
When the set $\st$ where the principal symbol
vanishes of second order is  involutive, the subprincipal symbol
${\sigma}_{sub}(P) = p_{m-1}$ is invariantly defined
at $\st$. In fact, on $\st$ it is equal to the {\em refined
principal symbol}, see \cite[Theorem~18.1.33]{ho:yellow}.

In the case where the principal symbol is real and vanishes of at
least second order  at 
the involutive manifold there are several results, mostly in the case
when the principal symbol is a product of real symbols of principal
type. Then the operator is not solvable if the imaginary part of the
subprincipal symbol has 
a sign change of finite order on a bicharacteristics of one the factors of
the principal symbol, see~ \cite{Ego}, \cite{Pop}, \cite{Wen1}
and~\cite{Wen2}.  

This necessary condition for solvability has been extended to some
cases when the principal symbol is real and vanishes of second order at the
involutive manifold. The conditions for solvability then involves the sign
changes of the imaginary part of the subprincipal symbol on the limits
of bicharacteristics 
from outside the manifold, thus on the leaves of the symplectic foliation of
the manifold, see~\cite{MU1}, \cite{MU2}, \cite{Men}
and~\cite{Yama2}. This has been extended to more general limit 
bicharacteristics of real principal symbols in~\cite{de:limit}.

When~$\st$ is not involutive, there are examples where the
operator is solvable for any lower order terms. For 
example when~ $P$ is effectively hyperbolic, then even the Cauchy problem
is solvable for any lower order term,
see ~\cite{ho:cauchy} and~\cite{Nishi}. 
There are also results in the cases when the principal symbol is a product of
principal type symbols not satisfying condition~ (${\Psi}$),
see~ \cite{CT},  \cite{GT},  \cite{Gold}, \cite{Treves} and~\cite{Yama1}.

In the present paper, we shall consider the case when the principal
symbol (not necessarily real valued) vanishes of at least second order at a
non-radial involutive manifold~$\st$. We shall assume that the
subprincipal symbol is   
of principal type with Hamilton vector field tangent to~$\st$ at the
characteristics, but transversal to the symplectic leaves of~$\st$. We
shall also assume that the subprincipal 
symbol is essentially constant on the symplectic leaves of~$\st$
by~\eqref{cond1}, and does not satisfying
condition (${\Psi}$), see Definition~\ref{subpsidef}. In the case when
the sign change is of infinite order, we will need a condition on the rate of
vanishing of both the Hessian of the principal symbol and the
complex part of the gradient of the subprincipal 
symbol on the semibicharacteristic of the subprincipal 
symbol, see
condition~\eqref{cond2}. Under these conditions, $P$ is not
solvable in a neighborhood of the semibicharacteristic, see
Theorem~\ref{mainthm} which is the main result of the paper. In this
case $P$ is an evolution operator, see~\cite{CCP} and~\cite{CPT} 
for some earlier results on the
solvability of evolution operators.

\section{Statement of results}

Let 
${\sigma}(P) = p  \in S^{m}_{hom}$ be the homogeneous principal symbol, we
shall assume that 
\begin{equation}
 {\sigma}(P) \ \text{vanishes of at least second order at ${\Sigma}_2$}
\end{equation}
where
\begin{equation}
 {\Sigma}_2 \quad \text{is a non-radial involutive manifold}
\end{equation}
Here non-radial means that the radial direction
$\w{{\xi},\partial_{\xi}}$ is not in the span of
the Hamilton vector fields of the manifold, i.e., not equal to $H_{f}$ on
${\Sigma}_2$ for some $f \in C^1$ vanishing at ${\Sigma}_2$. 
Then by a change of homogeneous symplectic coordinates we may assume that 
\begin{equation}\label{sigma2} 
 {\Sigma}_2 = \set{{\xi}' = 0} \qquad{\xi} = ({\xi}', {\xi}'') \in
 \br^k\times \br^{n-k}
\end{equation}
for some $k > 0$,
this can be achieved by conjugation by elliptic Fourier integral operators.
If $P$ is of principal type near $\st$ then, since solvability
is an open property, we find that a necessary condition for $P$ to be
solvable at $\st$ 
is that condition $({\Psi})$ for the principal symbol is satisfied in some
neighborhood of $\st$. Naturally, this condition is empty
on $\st$ where we instead need some conditions on the subprincipal symbol
\begin{equation} 
p_s = p_{m-1} + \frac{i}{2}\sum_{j}^{}\partial_{x_j}\partial_{{\xi}_j}p
\end{equation} 
which is equal to $p_{m-1}$ on $\st$ and invariantly defined  as a
function on $\st$ under  
symplectic changes of coordinates and conjugation with elliptic
pseudodifferential operators. 
(In the Weyl quantization, the subprincipal symbol is equal to $ p_{m-1}  $.)
When composing~$P$ with an elliptic
pseudodifferential operator $C$, the value of the subprincipal
symbol of $CP$ is equal to $cp_s + \frac{i}{2}H_p c = cp_s$
at $\st$, where $c = {\sigma}(C)$. 
Observe that the subprincipal symbol is complexly
conjugated when taking the adjoint of the operator.

Let  $T{\Sigma}_2^{\sigma}$ be the symplectic polar to $T
{\Sigma}_2$, which spans the symplectic leaves of $\st$. If $\st =
\set{{\xi}'= 0}$, 
$x= (x',x'') \in \br^k\times \br^{n-k}$, then the leaves are
spanned by  $\partial_{x'}$. Let
\begin{equation}
 T^{\sigma}{\Sigma}_2 = T{\Sigma}_2/T{\Sigma}_2^{\sigma}
\end{equation}
which is a symplectic space over ${\Sigma}_2$ which in these
coordinates is given by
\begin{equation}\label{tsigmadef}
T^{\sigma}\st =  
\set{((x_0,0,{\xi}_0''); (0,y'',0,{\eta}'')):\ (y'',{\eta}'') \in T^*\br^{n-k}}
\end{equation} 
Next, we are going to study the Hamilton vector field $H_{p_{m-1}}$ at $\st$.
If  $H_{p_{m-1}} \subseteq T {\Sigma}_2$ at $\st$ then we find that  
$ d p_s$ vanishes on $T{\Sigma}_2^{\sigma}$ so $dp_s$
is well defined on $T^{\sigma}{\Sigma}_2 $. 
In fact, $p_s = p_{m-1}$ on~$\st$ so if we choose coordinates so
that~\eqref{sigma2} holds, then $H_{p_{m-1}} \subseteq T {\Sigma}_2 $ is
equivalent to 
\begin{equation}\label{hsub}
H_{p_{m-1}}{\xi}'  = -\partial_{x'}p_{m-1} = -\partial_{x'}p_s = 0
\qquad \text{when ${\xi}'=0$} 
\end{equation}
which is invariant under
multiplication with non-vanishing factors when $p_s = 0$. 
Let $H_{p_s}$ be the Hamilton vector field of $p_s$
with respect the symplectic structure on the symplectic manifold
$T^{\sigma}\st$. In the chosen coordinates we have
$$
H_{p_s} = 
\partial_{{\xi}''}p_s \partial_{x''} - \partial_{x''}p_s
\partial_{{\xi}''}
$$ 
modulo $\partial_{x'}$, which is non-vanishing if
$\partial_{x''{\xi}''}p_s  \ne 0$. Since $p_s = p_{m-1}$ on~$\st$,
the difference between $H_{p_{m-1}}$ 
and $H_{p_s}$ is tangent to the leaves of ~$\st$.
Actually, since the subprincipal symbol is only well defined on $ \st $, 
the vector field $H_{p_s}$ is only well-defined up to terms tangent to the leaves.

Because of that,  we would need
that the subprincipal symbol ~$p_s$ is constant on the leaves of~$\st$, but 
that condition is {\em not} invariant under multiplication with
non-vanishing factors when $p_s \ne 0$. Instead
we shall use the following invariant condition:
\begin{equation}\label{cond1}
\left|d p_s\restr {TL} \right| \le C_0 |p_s| 
\end{equation}
for any leaf $L$ of $\st$. 
Then $ p_s $ is constant on the leaves modulo non-vanishing factors, 
according to the following lemma.

\begin{lem}\label{constsublem}
        If  $ d p_s\restr{T^{\sigma}{\Sigma}_2} \ne 0  $, then condition~\eqref{cond1} is equivalent to the
        fact that $p_s$ is 
        constant on the leaves of~$\st$ after multiplication with a
        smooth non-vanishing factor. Thus, if $\st = \set{{\xi}'=0}$ then~\eqref{cond1} gives
        $p_s(x,0,{\xi}'') = c(x,{\xi}'')q(x'',{\xi}'')$ with $0 \ne c \in C^\infty$. 
\end{lem}

\begin{proof}
        Choose coordinates so that $\st = \set{{\xi}'= 0}$. If $p_s \ne 0$ at a point
        $w_0 \in \st$ then~\eqref{cond1} gives that $\partial_{x'}\log p_s$ is
        uniformly bounded near $w_0$, where $\log p_s$ is a branch of the
        complex logarithm. Thus, by integrating with respect
        to~$x'$ in a simply connected neighborhood starting at $x' = x'_0$ 
        we find that
        \begin{equation}\label{factcond}
        p_s(x,0,{\xi}'') =
        c(x,{\xi}'')q(x'',{\xi}'')
        \end{equation}
        where $q(x'',{\xi}'') = p_s(x_0', x'',0,{\xi}'')\in C^\infty$ so  $0 \ne
        c \in C^\infty$ satisfies $c(x'_0,x'',{\xi}'') = 1$. When $p_s = 0$ we
        find that $d \re zp_s 
        \restr {T^{\sigma}\st} \ne 0$ for some $z \in \bc\setminus 0$
        by assumption. Thus we obtain locally that 
        \begin{equation*}
        p_s(x,0,{\xi}'') =
        c_\pm(x,{\xi}'')q_\pm(x'',{\xi}'') \qquad \text{on $S_\pm = \set{\pm
                        \re zp_s(x,0,{\xi}'') > 0}$}
        \end{equation*}
        where $q_\pm (x'',{\xi}'') = p_s(x_0',x'',0,{\xi}'')$, $0 \ne c_\pm
        \in C^\infty$ and $c_\pm(x'_0,x'',{\xi}'') = 1$ on $S_\pm$. Then we find that
        $p_s^{-1}(0)$ is independent of $x'$ and 
        $$
        \partial_{x''}^{\alpha}\partial_{{\xi}''}^{\beta} q_\pm(x'',{\xi}'')
        = \partial_{x''}^{\alpha}\partial_{{\xi}''}^{\beta}
        p_s(x_0',x'',0,{\xi}'') \qquad \forall\, {\alpha}\, {\beta} \qquad
        \text{on $S_\pm$} 
        $$ 
        so by taking the limit 
        at $S = \set{\re zp_s= 0}$ we find that the functions $q_\pm$ extend to
        $q \in C^\infty$.  
        Since $c_+ q = c_- q = p_s$ at $S$ we find that
        $c_+ = c_-$ at $S$ when $q \ne 0$. When $q = 0$ at $S$ we
        may differentiate in the normal direction of $S$ to obtain that
        $c_-\partial_{\nu} q = c_+ \partial_{\nu} q$ and since $\partial_{\nu} q
        \ne 0$ the functions $c_\pm$ extend to a
        continuous function $c$. By differentiating and taking the limit we
        find that 
        \begin{equation*}
        \nabla c_- q + c\nabla q = \nabla p_s = \nabla c_+ q + c\nabla q
        \qquad\text{at $S$} 
        \end{equation*}
        which similarly gives that $\nabla c_- = \nabla c_+$ at $S$ so $c
        \in C^1$. By repeatedly differentiating $c_\pm q$ we find by induction
        that $c \in C^\infty$, so
        we get smooth quotients ~$c$ and~ $q$ in~\eqref{factcond}. 
\end{proof}

Now, a
semibicharacteristic of $p_s$ will be a bicharacteristic of $\re ap_s$ on
$T^{\sigma}\st$, $C^\infty \ni a \ne 0$, with the natural orientation.
Observe that condition~\eqref{hsub} is  only  invariant under
multiplication with non-vanishing factors when $p_s = 0$. 

\begin{defn}
We say that the operator $P$ is of\/ {\em subprincipal type} if the
following hold when $p_s = 0$ on $\st$: $H_{p_{m-1}}\restr{{\Sigma}_2}
\subseteq T {\Sigma}_2$,  
\begin{equation}\label{subprinc}
 d p_s\restr{T^{\sigma}{\Sigma}_2} \ne 0 
\end{equation}
and the corresponding Hamilton vector field  $H_{p_s}$ of ~\eqref{subprinc}
does not have the radial direction. We call  $H_{p_s}$  the {\em
subprincipal Hamilton vector field} and the (semi)bicharacteristics are
called the {\em subprincipal 
(semi)bicharacteristics} on ${\Sigma}_2$. 
\end{defn}

Clearly, if~\eqref{sigma2} holds, then the condition that the Hamilton
vector field does not have the radial direction means that $
\partial_{\xi''}p_s \ne 0 $ 
or $\partial_{x''}p_s \not \parallel {\xi}''$ when $p_s = 0$ on
${\Sigma}_2  = \set{{\xi}'=0}$. 

In the case when the principal symbol ~$p$ is real, then a necessary
condition for solvability of the operator is that the
imaginary part of the subprincipal symbol does not change 
sign from $-$ to $+$ when going in the positive direction on a $ C^\infty $ limit
of normalized bicharacteristics of the principal symbol~$p$ at $ \st $,
see~\cite{de:limit}. When $p$ vanishes of exactly second order on~ $\st  = \set{{\xi}'=0}$ such limit
bicharacteristics are tangent to the leaves of~ $\st$. 
In fact, then Taylor's formula gives $ H_p =  \w{B\xi', \partial_{x'}} + \Cal O(|\xi'|^2)$ where $ B \ne 0 $, 
so the normalized Hamilton vector fields have limits that are tangent to the leaves.
When the principal symbol is proportional to a real valued valued symbol, this
gives examples of non-solvability when the subprincipal symbol is not
constant on the leaves of~$\st$. Thus condition~\eqref{cond1} is essential if 
there are no other conditions on the principal symbol.

\begin{rem}\label{realvalrem}
If $p_s$ is real valued, then by the proof of Lemma~\ref{constsublem}
it follows from~\eqref{cond1} that $p_s$ has constant sign on the
leaves of $\st$, since then $c > 0$ in~\eqref{factcond}.
\end{rem}

\begin{defn}\label{subpsidef}
We say that $P$ satisfies condition Sub(${\Psi}$) if 
$\im a p_s$ does not change sign from $-$ to $+$ when going in the 
positive direction on the subprincipal bicharacteristics of $\re
a p_s$  for any $0 \ne a \in C^\infty$. 
\end{defn}

Thus, condition Sub(${\Psi}$) is condition (${\Psi}$) given by~\eqref{psicond} on the subprincipal symbol $ p_s $.
Observe that since $p_s$ is only defined on $\st$, the Hamilton vector
field $H_{p_s}$ is only well defined in $T^{\sigma}{\Sigma}_2 =
T{\Sigma}_2/T{\Sigma}_2^{\sigma}$, thus it is well defined modulo
$\partial_{x'}$. But if~\eqref{cond1} holds then we find that
Sub(${\Psi}$) is a condition on~$p_s$ with respect to the 
symplectic structure of~$T^{\sigma}{\Sigma}_2$.
In fact,  by the invariance of  condition (${\Psi}$) given by~\cite[Lemma 26.4.10]{ho:yellow} 
condition Sub(${\Psi}$)
holds for any $a \ne 0$ such that $H_{\re ap_s} \ne 0$, so we may assume
by Lemma~\ref{constsublem} that $p_s$ is constant on the leaves of~$\st$.

Since condition Sub(${\Psi}$) is invariant under symplectic
changes of variables and multiplication with non-vanishing functions,
it is invariant under conjugation of the operator by elliptic
Fourier integral operators. Observe that the sign
change is reversed when taking the adjoint of the operator. 

Recall that the Hessian of the principal symbol $\hess p$ is the 
quadratic form given by $d^2 p$ at $\st$ which is defined on
the normal bundle $N \st$, since it vanishes on $T\st$.
By the calculus $\hess p$ is invariant, modulo non-vanishing
smooth factors, under symplectic 
changes of variables and multiplication of $P$ with elliptic
pseudodifferential operators.  

Next, we assume that condition  Sub(${\Psi}$) is not satisfied on a
semibicharacteristic ${\Gamma}$ of $p_s$, i.e., $\im ap_s$ changes
sign from $-$ to $+$ on the positive flow-out of $H_{\re
a p_s} \ne 0$ for some  $0 \ne a \in C^\infty$. 
Now if the sign change is {\em not} of finite order, 
we shall also need an extra condition on the rate of
vanishing of both the Hessian of the principal symbol and the complex
part of the gradient of the subprincipal 
symbol on the subprincipal semibicharacteristic. 
Then, we shall assume that there exists $C > 0$,
${\varepsilon} > 0$ and $0 \ne a \in C^\infty$ so that $d\re
ap_{s}\restr{T\st} \ne 0$ and  
\begin{equation}\label{cond2}
\mnorm{\hess p} + |d p_{s}\wedge d \ol p_{s}|  \le C |p_{s}|^{\varepsilon} 
\qquad \text{when $\re a  p_{s}= 0$ on $\st$}
\end{equation}
near ${\Gamma}$. Since~\eqref{cond2}
also holds for smaller~${\varepsilon}$ and larger~$C$, it is no
restriction to assume ${\varepsilon} \le 1$.
The motivation for~\eqref{cond2} is to prevent the
transport equation~\eqref{transexp} to disperse 
the support of the solution before the sign change of the imaginary
part of the subprincipal symbol localizes
it, see Remark~\ref{condrem}. 
We also find that $\nabla
 p_{s}$ is proportional to a real vector when
$ p_{s} = 0$ since then $d p_{s}\wedge d \ol p_{s} = 0$.
Then $ a $ is well-defined up to real factors.

\begin{rem}\label{condrem0} 
Condition~\eqref{cond2} is
invariant under multiplication of $P$ with elliptic pseudodifferential
operators and symplectic 
changes of coordinates. If~\eqref{cond1} also holds, then we obtain
that 
\begin{equation}\label{cond2o}
 \mnorm{d\hess p\restr{TL} }  \le C_1 |p_{s}|^{\varepsilon/2} 
\end{equation}
for any leaf $L$ of $\st$ when $\re a  p_{s}= 0$ near ${\Gamma}$.
\end{rem}

In fact, multiplication with an elliptic pseudodifferential operator
with principal symbol $c$ changes the principal 
symbol into $cp$, the Hessian of the principal symbol into $c\hess
p$ and the subprincipal symbol into  
$$
c  p_{s} + \frac{i}{2}H_p c \qquad \text{at $\st$} 
$$
where the last term vanishes at~$\st$ and contains the factor $\hess
p$, modulo terms vanishing of second order at~$\st$. Now we have that
\begin{equation*}
 |d cp_{s}\wedge d \ol {cp}_{s}| \le |c|^2|d p_{s}\wedge d \ol
 p_{s}| + C|p_{s}|
\end{equation*}
Thus we find that ~\eqref{cond2} holds with
$p$ replaced by $cp$, $p_{s}$ replaced by $cp_{s}$ and $a$
replaced with $a/c$. If~\eqref{cond1} also holds and 
we choose coordinates so that $\st =
\set{{\xi}'=0}$, then we obtain from Lemma~\ref{constsublem} that
$|p_s(x',x'',0,{\xi}'')| \cong |p_s(x'_0,x'',0,{\xi}'')|$ when $|x' -
x'_0| \le c$. Thus ~\eqref{cond2} gives
\begin{equation*}
 \mnorm{\hess p(x',x'',0,{\xi}'')} \le C_2
 |p_{s}(x'_0,x'',0,{\xi}'')|^{\varepsilon}  \qquad\text{ when $|x' -
x'_0| \le c$}
\end{equation*}
To show~\eqref{cond2o} it suffices to consider an element
$b_{jk}(x',x'',0,{\xi}'')$ of
$\hess p$. Clearly $|b_{jk}| \le  \mnorm{\hess p}$ so by adding
$C_2|p_s(x_0,x'',0,{\xi}'')|^{\varepsilon}$, we obtain that
$$
0 \le b_{jk}(x',x'',0,{\xi}'') \le 
2C_2|p_s(x_0,x'',0,{\xi}'')|^{\varepsilon} \qquad\text{ when $|x' -
x'_0| \le c$}
$$ 
Then we find that
$$
| \partial_{x'}b_{jk}(x_0,x'',0,{\xi}'')| \le C\sqrt{
  b_{jk}(x_0,x'',0,{\xi}'')} \le C' |p_s(x_0,x'',0,{\xi}'')|^{\varepsilon/2}
$$ 
by~\cite[Lemma~7.7.2]{ho:yellow}.

We shall study the microlocal solvability of the operator, which is
given by the following definition. Recall that $H^{loc}_{(s)}(X)$ is
the set of distributions that are locally in the $L^2$ Sobolev space
$H_{(s)}(X)$. 

\begin{defn}\label{microsolv}
If $K \subset S^*X$ is a compact set, then we say that $P$ is
microlocally solvable at $K$ if there exists an integer $N$ so that
for every $f \in H^{loc}_{(N)}(X)$ there exists $u \in \Cal D'(X)$ such
that $K \bigcap \wf(Pu-f) = \emptyset$. 
\end{defn}

Observe that solvability at a compact set $M \subset X$ is equivalent
to solvability at $S^*X\restr M$ by~\cite[Theorem 26.4.2]{ho:yellow},
and that solvability at a set implies solvability at a subset. Also,
by~\cite[Proposition 26.4.4]{ho:yellow} the microlocal solvability is
invariant under conjugation by elliptic Fourier integral operators and
multiplication by elliptic pseudodifferential operators.
We can now state the main result of the paper.

\begin{thm}\label{mainthm}
Assume that $P \in {\Psi}^m_{cl}(X)$ has principal symbol that
vanishes of at least second order at a non-radial involutive manifold, is of
subprincipal type,  does not satisfy
condition Sub(${\Psi}$) on the subprincipal semibicharacteristic
${\Gamma} \subset T^*X$, and satisfies~\eqref{cond1} near~$ \Gamma $.
In the case the sign change in Sub(${\Psi}$) 
is of infinite order we also assume condition~\eqref{cond2} near~$ \Gamma $. 
Then $P$ is not locally solvable at\/~${\Gamma}$.
\end{thm}

\begin{exe}\label{MU}
Let
\begin{equation}\label{MUeq}
 P = D_1D_2 + B(x,D_x)
\end{equation}
with $B \in {\Psi}^1_{cl}$, then ${\sigma}(B)$ is the subprincipal
symbol on ${\Sigma}_2 = \set{{\xi}_1 = {\xi}_2 = 0}$. Mendoza and
Uhlmann  proved in~\cite{MU2} that $P$ was not solvable if
$\im {\sigma}(B)$ changed sign as $x_1$ or $x_2$ increases on
${\Sigma}_2$, and they proved in~\cite{MU1} that $P$ was 
solvable if\/ $\im {\sigma}(B) \ne 0$ on ${\Sigma}_2$.
From this it is
natural to conjecture that the condition for solvability of $P$
is that $\im {\sigma}(B)$
does not change sign on the leaves of ${\Sigma}$, which are foliated by
$\partial_{x_1}$ and $\partial_{x_2}$. But the following is a counterexample
to that conjecture. Let
\begin{equation}
 P = D_1D_2 +  D_t + if(t,x,D_x) 
\end{equation}
with real and homogeneous $f(t,x,{\xi})\in S^1_{hom}$ satisfying
$\partial_{x_j}f = \Cal O(|f|)$ for $j=1$, $2$.
This operator is of subprincipal type and satisfies~\eqref{cond1}.
Then Theorem~\ref{mainthm} gives that $P$ is not solvable if $t \mapsto
f(t,x,{\xi})$ changes sign of finite order from $-$ to $+$, but
observe that $f$ has constant sign on  
the leaves of\/ ${\Sigma}_2$ by Remark~\ref{realvalrem}. Thus the
solvability of the operator $P$ in~\eqref{MUeq}
also depends on the real part of the subprincipal symbol at~$\st$.
In fact,  with the above conditions one can prove that $ D_1D_2 + if(t,x,D_x)  $ is solvable.
\end{exe}

\begin{exe}\label{LNS}
The linearized Navier-Stokes equation 
\begin{equation}
\partial_t u + \sum_j a_j(t,x) \partial_{x_j}u + \Delta_x u = f   \qquad a_j(x) \in C^\infty
\end{equation}
is of subprincipal type. The symbol is
\begin{equation}
i\tau + i\sum_j a_j(t,x)\xi_j -  | \xi |^2
\end{equation}
so  the subprincipal symbol is proportional to a real symbol on~$ \st = \set{   \xi = 0} $. Thus condition Sub(${\Psi}$) 
is satisfied.
\end{exe}

Now let $S^*M
\subset T^*M$ be the cosphere bundle where $|{\xi}| = 1$, and let
$\mn{u}_{(k)}$ be the $L^2$ Sobolev norm of order $k$, $u \in C_0^\infty$.
In the following, $P^*$ will be the $L^2$ adjoint of $P$.
To prove Theorem~\ref{mainthm} we shall use the following result.

\begin{rem}\label{solvrem}
If $P$ is microlocally solvable at ${\Gamma}\subset S^*\br^n$,
then Lemma 26.4.5 in~\cite{ho:yellow} gives that for any $Y \Subset
\br^n$ such that ${\Gamma} \subset S^*Y$ there exists an integer ${\nu}$
and a pseudodifferential operator $A$ so that
$\wf(A) \cap {\Gamma} = \emptyset$ and
\begin{equation}\label{solvest}
 \mn {u}_{(-N)} \le C(\mn{P^*{u}}_{({\nu})} + \mn {u}_{(-N-n)} +
 \mn{Au}_{(0)}) \qquad u \in C_0^\infty(Y)
\end{equation}
where~$N$ is given by Definition~\ref{microsolv}.
\end{rem}

We shall prove Theorem~\ref{mainthm} in Section~\ref{pfsect} by
constructing localized  
approximate solutions to $P^*u \cong 0$ and
use~\eqref{solvest} to show that $P$ is 
not microlocally solvable at~ ${\Gamma}$.
We shall first find a normal form for the adjoint operator.

\section{The normal form} \label{normform}

Assume that $P^*$ has the symbol expansion $p_m +
p_{m-1} + \dots$ where $p_j \in S^j_{hom}$ is homogeneous of degree~$j$. By
multiplying  $P^*$ with an elliptic pseudodifferential operator, we may
assume that $m=2$. 
Choose local symplectic coordinates $(t,x,y,{\tau},{\xi},{\eta})$ so that $\st =
\set{{\eta} = 0}$, which is foliated by leaves spanned by $\partial_y$.
Since $p_2$ vanishes of at least second order at $\st$ we find that 
\begin{equation*}
 p_2(t,x,y,{\tau},{\xi},{\eta}) = \sum_{jk}
 B_{jk}(t,x,y,{\tau},{\xi},{\eta}) {\eta}_j{\eta}_k
\end{equation*}
where $B_{jk}$ is homogeneous of degree $0$, $\forall\, jk$.

The differential inequality~\eqref{cond1} in these coordinates means that
$|\partial_{y}p_1| \le C|p_1|$ when ${\eta}= 0$, which by
Lemma~\ref{constsublem} gives that 
\begin{equation*} 
 p_1(t,x,y,{\tau},{\xi},0) =
 q(t,x,y,{\tau},{\xi})r_1(t,x,{\tau},{\xi})
\end{equation*}
near ${\Gamma}$, where $q$ is a non-vanishing smooth homogeneous function.
By multiplying with a pseudodifferential
operators with principal symbol equal to $q^{-1}$ on~$\st$ we may
assume that $q \equiv 1$ and 
that $p_1$ is constant on the leaves of $\st$ near ${\Gamma}$.
The Hamilton vector field of $p_1$ is then tangent to $\st$ by~\eqref{hsub}.

We have assumed that $P$ does not satisfy condition Sub(${\Psi}$) on a
semibicharacteristic ${\Gamma}$ of $p_1$ on $ \st $. Since we are now considering
the adjoint $P^*$ this means that  $\im ap_1$ changes
sign from $+$ to $-$ on the flow-out~${\Gamma}$ of $H_{\re
a p_1}$ on $\re a p_1^{-1}(0)$ for some  $0 \ne a \in C^\infty$. By the
invariance of condition Sub(${\Psi}$) given by  by~\cite[Lemma
26.4.10]{ho:yellow}, it is no restriction to assume
that $a$ is homogeneous and constant in ~$y$. By multiplication with
an elliptic pseudodifferential operator having principal symbol
$a^{-1}$ we may assume that $a \equiv 1$. Since $\im p_1$ changes
sign on ${\Gamma}$ there is a maximal semibicharacteristic ${\Gamma}' \subset
{\Gamma}$ on which $\im p_1 = 0$. Here
${\Gamma}'$ could be a point, which is always the case if the sign
change is of finite order.

Since $P$ is of subprincipal type we find that
$\partial_{t,x,{\tau},{\xi}}\re p_1 \ne 0$ on ${\Gamma}'$ by 
~\eqref{subprinc} so ${\Gamma}'$ is transversal to the leaves of
$\st$. Since $\im p_1\restr{\Gamma}$ has opposite signs near the  
boundary of ${\Gamma}'$, we may shrink ${\Gamma}$ so that
it is not a closed curve.
Since $ H_{\re p_1} $ is tangent to $ \st $ we can complete ${\tau}= \re p_1$ to a symplectic coordinate
system in a convex neighborhood of ${\Gamma}'$ so that ${\eta}= 0$ on
$\st$. In fact, this is obtained by solving the equation $H_{\tau}
{\eta} = 0$ with initial value on a 
submanifold transversal to $H_{\tau}$.
The change of variables can be then done by conjugation with suitable elliptic
Fourier integral operators. 

Now, by 
using Malgrange's preparation theorem in a neighborhood of ${\Gamma}'$
in $\st$ we find that
\begin{equation*}
 p_1(t,x,y,{\tau},{\xi},0) = q(t,x,{\tau},{\xi})({\tau}
 + r(t,x,{\xi})) \qquad q \ne 0 
\end{equation*}
near ${\Gamma}$, since $p_1$ is constant on the leaves of $\st$.
In fact, on ${\Gamma}'$ we have that $p_1 = 0$ and $ dp_1 \ne 0 $, so the division can
be done locally and by a partition of unity 
globally near ${\Gamma}$ after possibly shrinking ${\Gamma}$. 
Then by using Taylor's formula on $p_1$ we find since $q \ne 0$ that 
\begin{equation}
  p_1(t,x,y,{\tau},{\xi},{\eta}) = q(t,x,{\tau},{\xi})\left({\tau} +
  r(t,x,{\xi}) + A(t,x,y,{\tau},{\xi},{\eta})\cdot {\eta}\right)
\end{equation}
By multiplying $P$ with an elliptic pseudodifferential operator, we
may again assume $q \equiv 1$. 
Since $p_2$ vanishes of second order at~$\st$, this only changes $A$
with terms which has $\hess p_2$ as a factor and terms that vanish at~$\st$. 

We can write $r = r_1 + ir_2$ and $A = A_1 + i A_2$ with real valued
$r_j$ and $A_j$, $j= 1$, $2$.
Now we may complete 
\begin{equation*}
\re p_1 = {\tau}+ r_1(t,x,{\xi}) + A_1(t,x,y,{\tau},{\xi},{\eta})\cdot {\eta} 
\end{equation*}
to a symplectic
coordinate system in a convex neighborhood of ${\Gamma}'$. Since
$H_{\re p_1} \in T\st$ at $ \st $ we may keep $\st = \set{{\eta} = 0}$,
which preserves the leaves of $\st$ on which $p_1$ is constant. Thus, we
find that
\begin{equation}\label{subnormform}
 p_1 = {\tau}+ if(t,x,{\xi}) + iA(t,x,y,{\tau},{\xi},{\eta})\cdot {\eta} 
\end{equation}
where $f = r_2$ and $A = A_2$ are real valued. 
We also find that
\begin{equation}\label{defgamma}
 {\Gamma} = \set{(t,x_0,y_0, 0, {\xi}_0,0)} \qquad t \in I
\end{equation}
where $I$ is an interval in~$\br$. 
The symplectic change of coordinates can be made by conjugation with
elliptic Fourier integral operators, which only changes $A$ with terms
having $\hess p_2$ as a factor and terms that vanish at~$\st$. Observe
that $A$ need not be real valued after these changes.

We have assumed
that condition Sub(${\Psi}$) is {\em not}
satisfied for $P$ on the subprincipal semibicharacteristic ${\Gamma}$. Thus 
the imaginary part of the subprincipal symbol of~$P^*$ on~$\st$
\begin{equation}
 t \mapsto  f(t,x_0,{\xi}_0)
\end{equation}
changes sign from $+$ to $-$ as $t$ increases on $I
\subset \br$. Similarly, we have that $f = 0$ on ${\Gamma}'$ where
${\Gamma}'$ is given by~\eqref{defgamma} with $I$ replaced by $I' \subset I$.
By reducing to {\em
minimal bicharacteristics} on which $ t \mapsto
f(t,x,{\xi})$ changes sign as in~ \cite[p. 75]{ho:nec}, we may assume that $f$ 
vanishes of infinite order on a bicharacteristic ${\Gamma}'$
arbitrarily close to the original bicharacteristic,
if ${\Gamma}'$ is not a point (see~\cite[Section~2]{Witt} for a
more refined analysis). If ${\Gamma}'$ is not a point then it is a
{\em one-dimensional bicharacteristic} by 
\cite[Definition~3.5]{ho:nec}, which means that the
Hamilton vector field on ${\Gamma}'$ is proportional to a real vector.

In fact, if $ f(a,x_0,{\xi}_0) > 0 > f(b,x_0,{\xi}_0)$ for some $a <
b$, then we can define
\begin{equation*}
 L(x,{\xi}) = \inf \set{t-s:\ a < s < t < b\quad \text{such that}\quad
     f(s,x,{\xi}) > 0 > f(t,x,{\xi})}
\end{equation*}
when $(x,{\xi})$ is close to $(x_0,{\xi}_0)$, and we put
$L_0 = \liminf_{(x,{\xi}) \to (x_0,{\xi}_0)} L(x,{\xi})$. Then for
every ${\varepsilon} > 0$ there exists an open neighborhood ~$V_{\varepsilon}
$ of $(x_0,{\xi}_0)$ such that the diameter of~ $V_{\varepsilon}$ is
less than~${\varepsilon}$ and $L(x,{\xi}) > L_0 - {\varepsilon}/2$
when $(x,{\xi}) \in V_{\varepsilon}$. By definition, there exists
$(x_{\varepsilon}, 
{\xi}_{\varepsilon}) \in V_{\varepsilon}$ and $a < s_{\varepsilon} <
t_{\varepsilon} < b$ so that $t_{\varepsilon} - s_{\varepsilon} < L_0
+ {\varepsilon}/2$ and
$f(s_{\varepsilon},x_{\varepsilon},{\xi}_{\varepsilon}) > 0 > 
f(t_{\varepsilon},x_{\varepsilon},{\xi}_{\varepsilon})$. Then it is
easy to see that  
\begin{equation}\label{mincharcond}
 \partial_x^{\alpha}\partial_{\xi}^{\beta} f(t, x_{\varepsilon},
 {\xi}_{\varepsilon}) = 0 \qquad \forall\,{\alpha}\,{\beta} \qquad
 \text{when} \quad s_{\varepsilon}+ 
 {\varepsilon} < t  < t_{\varepsilon} - {\varepsilon} 
\end{equation}
since else we would have a sign change in a smaller interval than $L_0
- {\varepsilon}/2$ in~$V_{\varepsilon}$. We may choose a sequence
${\varepsilon}_j \to 0$ so that $s_{{\varepsilon}_j} \to s_0$ and
$t_{{\varepsilon}_j} \to t_0$, then $L_0 = t_0 - s_0$
and~\eqref{mincharcond} holds at~$(x_0,{\xi}_0)$ for $s_0 < t < t_0$.

We also obtain the following condition from~\eqref{cond2}.

\begin{rem}\label{condrem}
If the sign change of $t \mapsto f(t,x,{\xi})$ is of infinite order on
${\Gamma}$, then we find from assumption~\eqref{cond2} that
\begin{equation}\label{cond2a}
\mn{\set{{B_{jk}}}_{jk}} + | A| + |df| \ls
|f|^{\varepsilon} \qquad \text{near ${\Gamma}$ on $\st$}
\end{equation}
for some ${\varepsilon} > 0$. 
Here $a \ls b$ (and $b \gs a$) means that $a \le Cb$ for some $ C > 0$.
\end{rem}

In fact, terms having $\hess p_2\restr \st = \set{B_{jk}}_{jk}$ as a factor
can be estimated by~\eqref{cond2}, so we may assume
that~\eqref{subnormform} holds with real~$A$. The subprincipal symbol
is equal to $p_s = p_1 + i\sum_{jk}\partial_{y_j}B_{jk}{\eta}_k$
modulo terms that are $\Cal O(|{\eta}|^2)$ so $p_s = p_1$ on~$\st$. 
By Remark~\ref{condrem0} and~\eqref{cond1} we can 
estimate the terms $\partial_{y_j}B_{jk}d{\eta}_k$ in $dp_s$ by replacing~
${\varepsilon}$ with ~${\varepsilon}/2$ in~\eqref{cond2}, so we may
replace $p_s$ by ~$p_1$ in the estimate. 
Let $0 \ne a = a_1 + ia_2$ with real valued $a_j$ in~\eqref{cond2}, so
that $d \re ap_1\restr{T\st} \ne 0$. We have $dp_1 = d{\tau} + i(df +
A d{\eta})$ on~$\st$ so  
\begin{equation*}
 |d p_{1}\wedge d \ol p_{1}| \cong |df| + |A| \qquad \text{on $\st$}
\end{equation*}
Thus we find from~\eqref{cond2} that $|df| + |A| = 0$ on ${\Gamma}'$.
Since $d \re ap_1\restr{T\st} \ne 0$ we find that
$a_1 \ne  0$ on ${\Gamma}'$. On $\st$ we have that $\re a p_1 = 
a_1{\tau} - a_2 f = 0$ when ${\tau}= a_2 f/a_1$. We obtain that
$$
\im ap_1 = a_2 {\tau} + a_1 f = |a|^2 f/a_1 \qquad \text{when $\re a
  p_1 = 0$ on $\st$ near ${\Gamma}'$}
$$
which gives~\eqref{cond2a} from~\eqref{cond2}.

We obtain the following normal form for these operators of subprincipal type: 
\begin{equation}
 P^* = D_t + F(t,x,y,D_t,D_x,D_y)
\end{equation}
where $F \sim F_2 + F_1 + \dots$ with homogeneous $F_j \in
C^\infty(\br,S^j_{hom})$. Here $F_2$ vanishes of at
least second order on ${\Sigma}_2 = \set{{\eta}=0}$, so we find by
Taylor's formula that 
\begin{equation}
 F_2(t,x,y,{\tau},{\xi},{\eta}) = B(t,x,y,{\tau},{\xi},{\eta})
 = \sum_{jk}^{}B_{jk}(t,x,y,{\tau},{\xi},{\eta}){\eta}_j{\eta}_k
\end{equation}
with homogeneous $B_{jk}$, then $\{B_{jk}\}_{jk}\restr \st = \hess
F_2(t,x,y,{\tau},{\xi},0)$. Also we have that $F_1$ vanishes on the
semibicharacteristic ${\Gamma}'$ and 
\begin{equation}
F_1(t,x,y,{\tau},{\xi},{\eta}) = if(t,x,{\xi}) +
A(t,x,y,{\tau},{\xi},{\eta})\cdot {\eta} 
\end{equation}
Here $f$ is real and homogeneous of degree 1 and $A\restr \st =
\partial_{\eta} F_1\restr \st$. We have that 
the principal symbol ${\sigma}(P^*) = F_2$, and the subprincipal  
symbol ${\sigma}_{sub}(P^*) = {\tau} + if$ on $\st$. 
Thus we obtain the following result.

\begin{prop}\label{prepprop}
Assume that $P$ satisfies the conditions in
Theorem~\ref{mainthm}. Then by conjugation with elliptic Fourier
integral operators and multiplication with an elliptic
pseudodifferential operator we may assume that
\begin{equation}
 P^* = D_t + F(t,x,y,D_t,D_x,D_y)
\end{equation}
microlocally near ${\Gamma} = \set{(t,x_0,y_0,0,{\xi}_0,0):\ t \in I} \subset \st$
where $S^2_{cl} \ni F \cong F_2 + F_1 + \dots$ with $F_j \in
S^j_{hom}$ is homogeneous of degree $j$ and
$$
F_2 =
\sum_{jk}^{}B_{jk}(t,x,y,{\tau},{\xi},{\eta}){\eta}_j{\eta}_k \in
S^2_{hom}
$$
vanishes of second order on $\st$. We may also
assume that
$$
F_1(t,x,y,{\tau},{\xi},{\eta}) =  if(t,x,{\xi}) +
A(t,x,y,{\tau},{\xi},{\eta}) \cdot {\eta}  
$$
is homogeneous of degree 1 and $f$ is real valued such that $t
\mapsto f(t,x_0,{\xi}_0)$ changes sign from $+$ to $-$ as $t$
increases on $I \subset \br$. 
If $f(t,x_0,{\xi}_0) = 0$ on a subinterval $I' \subseteq I$ such that
$|I'| \ne 0$, then we may assume that $\partial_{t}^{k} \partial_{x}^{{\alpha}}
\partial_{{\xi}}^{{\beta}}f(t,x_0,{\xi}_0) = 0$, $\forall\, k\,
{\alpha}\, {\beta}$, for $t \in I'$.
If the sign change of $f$ is of infinite order 
then~\eqref{cond2a} is satisfied near~${\Gamma}$.
\end{prop}

For the proof of Theorem~\ref{mainthm} we shall modify the Moyer-H\" ormander
construction of approximate solutions of the type  
\begin{equation}
 u_{\lambda}(t,x,y) = e^{i{\lambda}{\omega}(t,x,y)}
 \sum_{j\ge 0} {\phi}_j(t,x,y) {\lambda}^{-j/N} \qquad {\lambda}\ge 1
\end{equation}
with $N$ to be determined later. 
Here the phase function ${\omega}(t,x)$ will be complex valued, but
$\im {\omega} \ge 0$ and $\partial \re {\omega} \ne 0$ when
$\im{\omega} = 0$.  
Letting $z = (t,x,y)$ we therefore have the formal expansion
\begin{equation}\label{trevesexp}
  p(z,D)  (\exp(i{\lambda}{\omega}){\phi})
  \sim \exp(i{\lambda}{\omega}) \sum_{{\alpha}}
   \partial_{{\zeta}}^{\alpha }
  p(z,{\lambda} \partial_z{\omega}(z))\Cal
  R_{\alpha}({\omega},{\lambda},D){\phi}(z)/{\alpha}! 
\end{equation}
where $\Cal R_{\alpha}({\omega},{\lambda},D){\varphi}(z) =
D_w^{\alpha}(\exp(i{\lambda} \wt
{\omega}(z,w)){\varphi}(w))\restr{w=z}$ 
and 
$$
\wt {\omega}(z,w) = {\omega}(w) - {\omega}(z) +
(z-w)\partial {\omega}(z)
$$
Observe that the values of the symbol are given by an
almost analytic extension, see Theorem  3.1 in Chapter VI and Chapter
X:4 in~\cite{T2}. This gives
\begin{multline}\label{exp0}
  e^{-i{\lambda}{\omega}}P^* e^{i{\lambda}{\omega}}{\phi} = \\
\big({\lambda}\partial_t{\omega}_{\lambda} +
   {\lambda}^2B(t,x,y,\partial_{t,x,y} {\omega}) +
   i{\lambda}f(t,x,\partial_x 
   {\omega}) -
   {\lambda}\partial_{\eta}^2B(t,x,y,\partial_{t,x,y} {\omega})
   \partial^2_{y}{\omega})/2 ){\phi}  \\ 
    + D_t{\phi} + {\lambda}\partial_{\eta}B(t,x,y,\partial_{t,x,y} {\omega})
   D_{y} {\phi} + \partial^2_{\eta}B(t,x,y,\partial_{t,x,y} {\omega})
   D_{y}^2 {\phi}/2 \\ + i\partial_{\xi} f(t,x,\partial_x
   {\omega}) D_x{\phi} 
+ A(t,x,y,\partial_{t,x,y} {\omega})D_y{\phi} +  \sum_{j \ge
  0}{\lambda}^{-j}R_{j}(t,x,y,D_{t,x,y}){\phi} 
\end{multline}
where $R_0(t,x,y) = F_0(t,x,y,\partial_{t,x,y}{\omega})$.
Here the values of the symbols at $(t,x,y,\partial_{t,x,y} {\omega})$ will
be replaced by finite Taylor expansions at $(t,x,y,\partial_{t,x,y} \re
{\omega})$. In fact, the almost analytic extensions are determined
by these Taylor expansions.

Because of the inhomogeneity coming from the terms of $B$, we shall
use a phase function ${\omega}(t,x)$ which is constant in~$y$, so that
\begin{equation}\label{udef}
 u_{\lambda}(t,x,y) = e^{i{\lambda}{\omega}(t,x)}
 \sum_{j\ge 0} {\phi}_j(t,x,y) {\lambda}^{-j/N} \qquad {\lambda}\ge 1
\end{equation}
When $\partial_y {\omega} \equiv 0$ the expansion~\eqref{exp0} becomes
\begin{multline}\label{exp}
  e^{-i{\lambda}{\omega}}P^* e^{i{\lambda}{\omega}}{\phi} =
\big({\lambda}(\partial_t{\omega}_{\lambda} +
     if(t,x,\partial_x 
   {\omega})){\phi} \\
    + D_t{\phi} + \partial^2_{\eta}B(t,x,y,\partial_{t,x} {\omega},0)
   D_{y}^2 {\phi}/2 + A(t,x,y,\partial_{t,x} {\omega},0)D_y{\phi} +
   i\partial_{\xi} f(t,x,\partial_x 
   {\omega}) D_x{\phi} \\
 +  \sum_{j \ge
  0}{\lambda}^{-j}R_{j}(t,x,y,D_{t,x,y}){\phi} 
\end{multline}
where $R_0(t,x,y) = F_0(t,x,y,\partial_{t,x}{\omega},0)$, and
$R_m(t,x,y,D_{t,x,y})$ are differential operators of order $j$ in $t$,
order $k$ in $x$ and order $\ell$ in $y$, where $j + k + \ell \le m+2$ for $m > 0$.
In fact, this follows since
$\partial_\tau^j\partial_\xi^{\alpha}\partial_\eta^{\beta} F_k \in S^{k-j
  -|{\alpha}| -|{\beta}|}$ by homogeneity.

\section{The Eikonal Equation}

We shall first solve the eikonal equation approximately, which is
given by the highest order term of~\eqref{exp} 
\begin{equation}\label{eikeq}
  \partial_t{\omega} + if(t,x,\partial_x{\omega}) = 0
\end{equation}
where $t \mapsto f(t,x,{\xi})$ changes sign from $+$ to $-$ for
some~$(x,{\xi})$ as $t$
increases in a neighborhood of ${\Gamma} = \set{(t,x_0,{\xi}_0): \ t \in
  I}$ on which $f(t,x,{\xi})$ vanishes. If $|I| \ne 0$ then by
reducing to minimal bicharacteristics as in Section~\ref{normform}, we
may assume that~$f$ vanishes of infinite order 
at~${\Gamma}$. We shall choose the phase
function so that  $\im {\omega} \ge 0$ and $\partial_x^2 \im{\omega} >
0$ near the interval. By changing coordinates, it is no
restriction to assume $0 \in I$. We shall use the approach by
H\" ormander~\cite{ho:nec} in the principal type case and use the
phase function to localize in~$t$ and~$x$. Observe that
since ${\omega}$ does not depend on ~$y$ the
localization in the $y$ variables will be done in the amplitude ${\phi}$.

We shall take the Taylor expansion of ${\omega}$ in $x$:
\begin{equation}\label{omegaexp}
 {\omega}(t,x) = w_0(t) +
 \w{x-x_0(t), {\xi}_0(t)} + \sum_{2 \le |{\alpha}|
   \le K}
 w_{{\alpha}}(t) (x-x_0(t))^{\alpha}/{\alpha}! 
\end{equation}
Here ${\alpha}= ({\alpha}_1, {\alpha}_2, \dots)$, with
${\alpha}_j \in \bn$,   ${\alpha}! = \prod_j {\alpha}_j!$ and
$|{\alpha}| = {\alpha}_1 + {\alpha}_2  + \dots$. 
Then we find that 
\begin{multline}\label{dtomega}
 \partial_t{\omega}(t,x) = w_0'(t) 
 - \w{x_0'(t) ,{\xi}_0(t) } +  \w{x-x_0(t), {\xi}'_0(t)} +
  \sum_{2 \le |{\alpha}| \le K} w_{\alpha}'(t)(x-x_0(t))^{\alpha}/{\alpha}!\\
  -  \sum_{\substack{1 \le |{\alpha}| \le K-1\\ k}} 
w_{\alpha + e_k}(t)(x-x_0(t))^{\alpha}x_{0,k}'(t)/{\alpha}!
\end{multline}
where $e_k = (0, \dots, 0,1,0,
\dots, 0)$ is the $k$:th unit vector. We
also find
\begin{equation}\label{dxomega}
  \partial_{x_j} {\omega}(t,x) = {\xi}_{0j}(t) +
  \sum_{1 \le |{\alpha}| \le K-1}   w_{{\alpha}+
   e_j}(t)(x-x_0(t))^{\alpha}/{\alpha}! \\ =
 {\xi}_{0j}(t)  +  {\sigma}_{j}(t,x) 
\end{equation}
Here ${\xi}_0(t) = ({\xi}_{0,1}(t), \dots)$ and ${\sigma}=
\set{{\sigma}_j}_j$ is a finite expansion in powers of
${\Delta}x = x-x_0$.
We define the value of $ f(t,  x,\partial_x{\omega})$ by the Taylor
expansion 
 \begin{multline}
 f(t,  x,\partial_x{\omega}) = f(t,x,{\xi}_0+ {\sigma}) \\ 
= f(t, x, {\xi}_0) + \sum_{j}^{}\partial_{\xi_j}f(t, x,{\xi}_0){\sigma}_{j}
 + \sum_{jk}^{}\partial_{\xi_j}\partial_{\xi_k} f(t, x,{\xi}_0)
 {\sigma}_{j}{\sigma}_{k}/2 + \dots
\end{multline}
Now the value at $x= x_0$ of~\eqref{eikeq} is equal to
$ 
 w_0'(t) -
 \w{x_0'(t),{\xi}_0(t)}  + i f(t, x_0(t), {\xi}_0(t)). 
$ 
This vanishes if
\begin{equation}\label{2a}
\left\{
\begin{aligned} 
&\re w'_0(t) =  \w{x_0'(t),{\xi}_0(t)}\\
&\im w_0'(t) = -f(t, x_0(t), {\xi}_0(t))
\end{aligned}
\right.
\end{equation}
so by putting $w_0(0) = 0$ this will determine $w_0$ once we have $(x_0(t) ,{\xi}_0(t) )$.

We shall simplify the notation and put $w_k =
\set{w_{\alpha}k!/{\alpha}!}_{|{\alpha}| = k}$ 
so that $w_k$ is a multilinear form.
The first order terms in $x-x_0$ of ~\eqref{eikeq}
vanish if
\begin{equation*}
{\xi}_0'(t) -  w_{2}(t)x_0'(t) + i
(\partial_x
f(t,x_0(t),{\xi}_0(t)) + \partial _{\xi}f(t,x_0(t),{\xi}_0(t))w_{2}(t)) = 0
\end{equation*}
We find by taking real and imaginary parts that
\begin{equation}\label{2}
\left\{
\begin{aligned} 
&{\xi}_0' = \re w_{2} x_0'  +
\partial _{\xi}f(t,x_0,{\xi}_0) \im w_{2}\\
&x_0' = (\im w_{2})^{-1}(\partial_x f(t,x_0,{\xi}_0) + \partial
_{\xi}f(t,x_0,{\xi}_0) \re w_{2})
\end{aligned}
\right.
\end{equation}
with $(x_0(0), {\xi}_0(0)) = (x_0, {\xi}_0)$, which will determine $x_0(t)$ and
${\xi}_0(t)$ if $|\im w_{2}| \ne 0$. 

The second order terms in $x-x_0$ vanish if
\begin{equation*}
w'_{2}/2 - w_{3}x_0'/2
+ i(\partial_{\xi}f w_{3}/2 + \partial_x^2f/2 + 
 \partial_x\partial_{\xi}f w_{2} + w_2 \partial_{\xi}^2f w_2/2)  = 0
\end{equation*}
which gives
\begin{equation}\label{w2eq}
 w'_{2}  =  w_{3}x_0'
- i(\partial_{\xi}f  w_{3}
 + \partial_x^2f  + 
2 \partial_x\partial_{\xi}f w_{2} + w_2 \partial_{\xi}^2f w_2) 
\end{equation}
with initial data $ w_2(0) $ so that 
$\im w_{2}(0) > 0$. 

We find that the terms of order $k > 2$ vanish if
\begin{equation}\label{wkeq}
w_{k}' -  w_{k+1}x'_0 = F_k(t,x_0,{\xi}_0,\set{w_j})
\end{equation}
where we may choose $w_k(0) = 0$. Here $F_k$ is a linear combination of
the derivatives of $f$ of order 
$\le k$ multiplied by polynomials in $w_j$ with $2 \le j \le k+1$. When
$k = K$ we get  $w_K' = F_K(t,x_0,{\xi}_0, \set{w_j})$ where $j \le K$. The
equations~\eqref{2}--\eqref{wkeq} form a quasilinear system of
differential equations, which can be solved in a convex neighborhood of~ $0$.
In the case when $|I| \ne 0$ we have assumed that
$\partial^{\alpha}_{t,x,{\xi}}f(t,x_0,{\xi}_0) \equiv 0$, $ \forall
{\alpha}$, for $t \in I$. Then we find from~\eqref{2}--\eqref{wkeq}
that  $x_0$, ${\xi}_0$ and $w_k$ are constant in ~$t \in I$,
so we may solve~\eqref{2}--\eqref{wkeq} in a
convex neighborhood of~ $I$. Observe that the lower order
terms cannot change the condition that $\im \partial_x^2{\omega} \ge c
> 0$ and $\im {\omega}(t,x) \ge 0$ if $|x-x_0(t)| \ll 1$. Summing up,
we have proved the following result. 

\begin{prop}\label{eikprop0}
Let\/ ${\Gamma} = \set{(t,x_0,{\xi}_0): \ t \in I}$ and assume that
$\partial_{t}^{k} 
\partial_{x}^{{\alpha}} \partial_{{\xi}}^{{\beta}}f(t,x_0,{\xi}_0) = 0$
for all $t \in I$ in the case  $|I|\ne 0$. Then we may solve~\eqref{eikeq}
with ${\omega}(t,x)$ given by~\eqref{omegaexp} in a convex neighborhood~
${\Omega}$ of\/
${\Gamma}$ modulo $\Cal O(|x-x_0(t)|^M)$, $\forall\, M$, such that $
(x_0(t), {\xi}_0(t)) = (x_0,{\xi}_0)$ when $t \in I$ and $w_k(t) \in
C^\infty$ such that $w_0(t) = 0$,
$\im w_2(t) > 0$ and $w_k(t) = 0$, $k > 2$, when $t \in
I$. 
\end{prop}

Then we obtain that $\im {\omega}(t,x) \ge c|x - x_0(t)|^2$ near ${\Gamma}$, $c
> 0$, so the errors that are $\Cal O(|x-x_0|^M)$ in the eikonal
equation will give 
terms that are bounded by $C_M{\lambda}^{-M/2}$.
But we have to show that $t \mapsto f(t,x_0(t), {\xi}_0(t))$
also changes sign from $+$ to $-$ as $t$ increases
for some choice of $ (x_0,{\xi}_0) $. 
This problem will be studied in the next section, with a special
emphasis on the finite vanishing case.
By~\eqref{2a} we then obtain that $t \mapsto \im
w_0(t)$ has a local minimum on ~$I$ 
which can be equal to~ $0$ by subtracting a constant.

\section{The Change of Sign}

We have assumed that condition Sub(${\Psi}$) for $P$ is {\em not} satisfied
near the subprincipal semibicharacteristic ${\Gamma} =
\set{(t,x_0,{\xi}_0): \ t \in I}$, so that $t \mapsto 
f(t,x,{\xi})$ changes sign from $+$ to $-$ for some~$(x,{\xi})$ as $t$
increases near~${\Gamma}$. But  
after solving the eikonal equation we have to know that  $t \mapsto
f(t,x_0(t),{\xi}_0(t))$ has the same sign change, possibly after changing
the starting point ~$(x_0,{\xi}_0)$. In order to do so we shall use the
invariance of condition Sub(${\Psi}$), but note that
condition~\eqref{cond2a} is only assumed when the change of
sign is of infinite order. Therefore we shall first consider the case when the
sign change is of {finite order} and show that this condition is 
preserved after solving the eikonal equation. Thus assume that 
\begin{equation}\label{finitezero}
 \partial_t^k f(t_0,x_0,{\xi}_0) < 0 \qquad \text{and} \qquad \partial_t^j
   f(t_0,x_0,{\xi}_0) = 0 \qquad \text{for } j < k
\end{equation}
for some odd integer $k$, where we may assume $t_0= 0$. Now, if the
order of the zero is not 
constant in a neighborhood of ~ $(x_0,{\xi}_0)$ then in any  
neighborhood the mapping
$t \mapsto f(t,x,{\xi})$ must have a zero of odd order with sign change from $+$ to
$-$, and the order of vanishing is constant almost everywher on  $f^{-1}(0)$. 
In fact, this follows since $ \partial_t^k f \ne 0 $, 
$t \mapsto f(t,x,{\xi})$ goes from $ + $ to $ - $ and
the set where the order of the zero changes is nowhere dense
in $f^{-1}(0)$ since it is the union of boundaries of closed sets in the relative
topology. By possibly changing $(t_0, x_0,{\xi}_0)$ we may assume that 
\eqref{finitezero} holds with $ t_0 = 0 $,  and 
that the order of the zero is odd and constant near~$(x_0,{\xi}_0)$, then the zeros
forms a smooth manifold by the implicit function theorem.
By using Taylor's formula, we find that
$f(t,w) = a(t,w)(t-t_0(w))^k$, where $k \ge 1$ is odd, $w=(x,{\xi})$,
$t_0(w_0) = 0$ and  $a < 0$ in a neighborhood of
$w_0 = (x_0,{\xi}_0)$. Then we find
\begin{equation}\label{dwf}
 \partial_w f = \partial_w a (t-t_0)^k - a k(t-t_0)^{k-1} \partial_w t_0
\end{equation}
which vanishes of at least order $k-1$ in $t$ at $f^{-1}(0)$. 
Let $w(t) = (x_0(t),{\xi}_0(t))$ then
 $$
 f(t,w(t)) = f(t,w_0) +
 \partial_{w}f(t, w_0){\Delta}w(t) + \Cal O(|{\Delta}w(t))|^2)
 $$
where ${\Delta}w(t) = w(t) - w_0$. Now $t \mapsto f(t, w_0)$ vanishes of order
$k$ in~$t$ at 0 and $t \mapsto \partial_{w}f(t, w_0)$ vanishes of at least
order $k-1$, so if $t \mapsto {\Delta}w(t)$ vanishes of at least order 
$k$ then by~\eqref{dwf} we find that $t \mapsto f(t,w(t))$ vanishes of order $k$. Since
$\frac{d}{dt}{\Delta}w(t) = w'(t)$, we will need
the following result.

\begin{lem}\label{vanlem}
Let $(x_0(t), {\xi}_0(t))$ be the solution to the
equation~\eqref{2} with $\im w_2(0) \ne 0$ and assume that $t \mapsto
\partial_{w}f(t,x_0, {\xi}_0)$ vanishes of order $r \ge 1$ at $t=0$. Then
 $(x_0'(t), {\xi}_0'(t))$ vanishes of order $r$ and $ \Delta w(t) $ 
 vanishes of order $r+1$ at $t=0$. 
\end{lem}

\begin{proof}
By \eqref{2} we have that 
\begin{equation}
 w'(t) = (x_0'(t),{\xi}_0'(t)) = A(t) \partial_wf(t,w(t))\qquad w(0) =
 w_0
\end{equation}
Here we have $|A(0)| \ne 0$ if $\im w_2(0) \ne 0$, in fact
$w'(0) = 0$ then gives $ \partial _{\xi}f(0, w_0) = 0
$ and $ \partial _x f(0, w_0) = 0$ by~\eqref{2}.

Now we denote ${\phi}_0(t) = \partial _wf(t,w_0)$ and ${\phi}_1(t) =
\partial _wf(t, w(t))$. Then we have that $w'(t) =
A(t){\phi}_1(t)$ and the condition is that ${\phi}_0(t)$ vanishes of order $r
\ge 1$ at 0. We shall proceed by induction, and first assume that $r =
1$. Since $w(0) = w_0$ we find ${\phi}_1(0) = {\phi}_0(0)= 0$ and thus
$w'(0) = 0$.   

Next, for $r > 1$ we assume by induction that
$w'(t)$ vanishes by order $r-1$ at 0 so $w^{(k)}(0) = 0$ for $k < r$, and
then we shall show that $w^{(r)}(0) = 0$ so that $w'$ vanishes of
order $r$. By
using the chain rule we obtain that
\begin{equation}
 \partial_t^r \left(g(t,w(t))\right) = \sum_{\substack{0 \le j \le r\\
  \sum_i r_i + j = r}}
 c_{j,{\alpha}}\partial_t^j\partial_w^{\alpha} g(t,w(t)) \prod_{i=1}^{|{\alpha}|}
 w^{(r_i)}(t)
\end{equation}
for any  $g(t,w) \in C^\infty$. Thus, for $g = \partial_w f$ we find that 
$${\phi}_1^{(k)}(0) = {\phi}_0^{(k)}(0) +
\partial_t^{k-1}\partial_w^2f(0,x_0,{\xi}_0)w'(0) + \dots + 
\partial_w^2f(0,x_0,{\xi}_0)w^{(k)}(0) =   {\phi}_0^{(k)}(0) = 0
$$
for $k < r$, since the other terms has some factor $w^{(j)}(0) = 0$,
$j \le k$,   
which implies that ${\phi}_1(t)$ vanishes for order $r$. Since $w' = A
{\phi}_1$ we find that $w'(t)$ vanishes of order $r$, 
which gives the induction step and the proof.
\end{proof}

Now, if $f(t,w_0)$ vanishes of order $k$ then
$\partial_wf(t,w_0)$ vanishes of order $k-1$. Thus $w'(t)$ vanishes of
order $k-1$  by Lemma~\ref{vanlem} and since $w(0) = w_0$ we find that
${\Delta}w(t)$ vanishes of order $k$. Thus, we find that $f(t,w(t)) -
f(t,w_0)$ vanishes of order $2k-1$, so these terms vanish of same order
if $k > 1$. In the case $k = 1$, we shall use an argument of
H\" ormander~\cite{ho:nec} for the principal type case. We obtain
from~\eqref{2a} that $\partial_t \left(f(t,w(t))\right) = -\im w_0''(t)$ thus
\begin{equation}
 \im w_0''(0) = -\partial_t f(0, w_0) -
 \partial_{\xi}f(0,w_0)\cdot {\xi}_0' -
 \partial_{x}f(0,w_0)\cdot x_0'
\end{equation}
where $\partial_tf(0,w_0) = -c < 0$. We find from~\eqref{2} that
\begin{equation}
\left\{
\begin{aligned} 
&{\xi}_0'(0) = \re  w_2(0)x_0'(0)  +  \partial _{\xi}f(0, w_0)\im  w_2(0) \\
& x_0'(0) = (\im  w_2(0))^{-1}(\partial_x f(0,w_0) + \partial
_{\xi}f(0,w_0) \re  w_2(0))  
\end{aligned}
\right.
\end{equation}
If $\partial _{\xi}f(0,w_0) = 0$ then we find that $
x_0'(0) =  (\im  w_2(0))^{-1}\partial_x f(0,w_0)$
and obtain
\begin{equation}
 \im w_0''(0) = c -   \partial_{x}f(0,w_0)(\im  w_2(0))^{-1}
\partial_x f(0,w_0)   > c/2 > 0
\end{equation}
by choosing $\im w_2(0) = {\kappa}\id$ with ${\kappa} \gg 1$. If
$\partial _{\xi}f(0,w_0) \ne 0$ then we may choose $\re
w_2(0)$ so that 
\begin{equation}
 \partial_x f(0,w_0) + \partial _{\xi}f(0,w_0) \re
 w_2(0) = 0
\end{equation}
Then we find $x_0'(0) = 0$ and we obtain
\begin{equation}
  \im w_0''(0) = c -
  \partial_{{\xi}}f(0,w_0)\im
  w_2(0)\partial_{{\xi}}f(0,w_0) > c/2 > 0
\end{equation}
by choosing $\im w_2(0) = {\kappa}\id$ with $0 < {\kappa} \ll 1$.
Thus in both cases we find that $\partial_t f(t,w(t)) = \im w_0'(t) < 0$ 
at $ t=0 $.

We find that $ t \mapsto f(t,w(t)) $ changes sign from $ + $ to $ - $ of order $ k $ as  $ t $ increases 
at $ t=0 $. We may then rewrite the equation as
\begin{equation}
 \im w_0'(t) = -t^kc(t)
\end{equation}
where $c(t) > 0$ in a neighborhood of the
origin. 
Since $\im w_2(0) > 0$ we find that  
\begin{equation}\label{1}
  e^{i{\lambda}{\omega}(t,x)} \le e^{-c_0{\lambda}(t^{k+1} + |x-x_0|^2)}
  \qquad |x-x_0| \ll 1 \qquad |t| \ll 1
\end{equation}
Thus the errors that are $\Cal O(|x-x_0|^M)$ in the eikonal equation will give
terms that are bounded by $C_M{\lambda}^{-M/2}$.

We shall also consider the case when $t \mapsto f(t,x,{\xi})$
changes sign from $+$ to $-$ of infinite order near~${\Gamma}$. 
If ${\Gamma}$ is not a point, then by reducing to a minimal
bicharacteristic as in Section~\ref{normform}, we
may assume that $f(t,x_0,{\xi}_0)$ vanishes of 
infinite order when $t \in I$ and $|I| \ne 0$. We then obtain an
approximate solution to the eikonal equation by
solving~\eqref{2}--\eqref{wkeq} with initial data $w = (x,{\xi})$ and
$w_k(0)$, $k \ge 2$, which gives a 
change of coordinates $(t,w) \mapsto (t,w(t))$. If in any neighborhood
of~${\Gamma} = \set{(t,x_0,{\xi}_0): \ t \in I}$ there exist points
in $f^{-1}(0)$ where $\partial_t f < 0$, then as before we can construct
approximate solutions in any neighborhood of~${\Gamma}$
satisfying~\eqref{1} with $k = 1$. 
If $\partial_t f \ge 0$ on $f^{-1}(0)$ in some neighborhood of~${\Gamma}$,
then by the invariance of  
condition (${\Psi}$) there will still exist a change of sign of $t
\mapsto f(t,w(t))$ from 
$+$ to $-$ in any neighborhood of~${\Gamma}$  after the change of
coordinates, see~\cite[Lemma~26.4.11]{ho:yellow}. (Recall that 
conditions~\eqref{cond1} and~\eqref{cond2} hold in some
neighborhood of ${\Gamma}$.) 
Thus if $F'(t) = -\im w'_0(t) = f(t,w(t))$ then $t \mapsto F(t)$ has a
local maximum 
at some $t= t_0$, and after subtraction the maximum can be assumed to be
equal to 0. 
By choosing suitable initial value
$(x_0,{\xi}_0)$ for~\eqref{2} at $t = t_0$ we obtain that
\begin{equation}\label{11}
  e^{i{\lambda}{\omega}(t,x)} \le e^{{\lambda}(F(t) - c|x-x_0|^2)}
  \qquad |x-x_0| \ll 1
\end{equation}
where $F'(t) = f(t,w(t))$ so that $\max_I F(t) = 0$ with $F(t) <
0$ when $t \notin I$ near $\partial I$.

\begin{prop}\label{eikprop}
Assume that $t \mapsto f(t,x_0,{\xi}_0)$ changes sign from $+$ to $-$
as $t$ increases near~$I$ and that $\partial_{t}^{k} 
\partial_{x}^{{\alpha}} \partial_{{\xi}}^{{\beta}}f(t,x_0,{\xi}_0) = 0$
for all $t \in I$ when $|I|\ne 0$. Then we may solve~\eqref{eikeq}
in a neighborhood~${\Omega}$ of\/
${\Gamma} = \set{(t,x_0,{\xi}_0): \ t \in I}$ modulo $\Cal
O(|x-x_0(t)|^M)$, $\forall\, M$, 
with ${\omega}(t,x)$ given by~\eqref{omegaexp} such that
the curve $t \mapsto (x_0(t), {\xi}_0(t))$, $t \in (t_1, t_2)$, is
arbitrarily close to~ ${\Gamma}$, $w_k(t) \in C^\infty$,
$\im w_2(t) \ge c > 0$ when $t \in (t_1, t_2)$, $\min_{(t_1, t_2)}
\im w_0(t) = 0$ and $\im w_0(t_j) = c > 0$, $j = 1$, $2$.
\end{prop}

Observe that since $\im w_0 \ge 0$ we find that
$f(t_0,x_0(t_0),{\xi}_0(t_0)) = -\im w_0'(t) = 0$ at a minimum~$t_0
\in (t_1, t_2)$.
As before, the errors that are $\Cal O(|x-x_0|^M)$ in the eikonal
equation will give terms that are bounded by $C_M{\lambda}^{-M/2} $,
$\forall\, M$.  
Observe that cutting off where $\im w_0 > 0$ will give errors that are
$\Cal O({\lambda}^{-M})$, $\forall\, M$.

\section{The Transport Equations}

Next, we shall solve the transport equations given by the following
terms in~\eqref{exp}:
\begin{multline}\label{transexp}
  D_t{\phi} + \partial_{\eta}^2B(t,x,y,\partial_{t,x}{\omega},0)
  D_y^2{\phi}/2   
+ A(t,x,y,\partial_{t,x}{\omega},0)D_y{\phi} 
   + i \partial_{\xi}f(t,x,\partial_x{\omega})D_x {\phi}
\\ + \sum_{j \ge 0}{\lambda}^{-j}R_j(t,x,y,D_{t,x,y}){\phi} 
\end{multline}
near ${\Gamma} = \set{(t,x_0,y_0,0,{\xi}_0,0):\ t \in I}$.
Here $R_0(t,x,y) = F_0(t,x,y,\partial_{t,x}{\omega},0)$ and when $m > 0$
we have that
$R_m(t,x,y,D_{t,x,y})$ are differential operators  of order $j$ in $t$,
order $k$ in $x$ and order $\ell$ in $y$, where $j + k + \ell \le m+2$.  
Assuming the conclusions in Proposition~\ref{eikprop} hold,
we shall choose suitable initial values of the amplitude ${\phi}$ at
$t = t_0$, which is chosen so that $\im w_0(t_0) = 0$.
Observe that the second order differential operator given by the first four terms in~\eqref{transexp} need not be solvable in general. Instead, by Lemma~\ref{intrem} we can treat the $D_x $ and $ D_y $ terms as perturbations, using condition~\eqref{cond2a} in the infinite vanishing case.  

Since the phase function ${\omega}(t,x)$ is complex valued, we will
replace the values of the symbols at $(\tau, \xi) = \partial_{t,x}
{\omega}(t,x)$ by finite Taylor expansions
at~$(\re w_0'(t),{\xi}_0(t))$. By~\eqref{dtomega} and 
~\eqref{dxomega} this will give expansions in
powers of $x-x_0(t)$ and $\im w'_0(t) = -f(t,x_0(t),
{\xi}_0(t))$. Then, we shall solve the transport equations up to
arbitrarily high powers of  $x-x_0(t)$ and $f$. 
Since the imaginary part of the phase function $\im {\omega} \ge 0$
vanishes of second order at $x = x_0(t)$ we will obtain 
by Lemma~\ref{intrem} below that this will give a
solution modulo any negative power of ${\lambda}$.

We shall use the amplitude expansion
\begin{equation}\label{phidef}
{\phi}(t,x,y) = \sum_{k \ge 0}^{} {\varrho}^{-k}{\phi}_k(t,x,y)
\end{equation}
and solve the transport equation recursively in $k$.
Here ${\phi}_k$ depends on ${\varrho}$ but with uniform bounds in a suitable symbol class, and
${\varrho} = {\lambda}^{1/N}$ with $N$ to be determined later. 
By doing the change of variables $(t,x,y) \mapsto (t-t_0,x-x_0(t),
y-y_0)$ we find 
that $D_t$ changes into $D_t -x_0'(t)D_x$ which does not change the
order of $R_j$ as differential operator. 
Thus we may assume $t_0 = 0$, $x_0(t) \equiv 0$ and $y_0 = 0$.

Next, we apply~\eqref{transexp} on ${\phi}$ given by~\eqref{phidef}. Since ${\varrho} =
{\lambda}^{1/N}$ we obtain the terms 
\begin{multline}\label{transeq0}
D_t {\phi} + A_0(t,x)D_x{\phi} + A_1(t,x,y)D_y{\phi} +
A_2(t,x,y)D_y^2{\phi}  \\ + \sum_{j \ge 0}{\varrho}^{-jN}
R_j(t,x,y,D_{t,x,y}){\phi}
\end{multline}
where 
$
A_0(t,x) = 
i \partial_{\xi}f(t,x,{\xi}_0(t)+ {\sigma}(t,x)) - x'_0(t) 
$,
\begin{equation}
A_1(t,x,y)= A(t,x,y,\partial_t{\omega}(t,x),{\xi}_0(t)+ {\sigma}(t,x),0)
\end{equation}
and
\begin{equation}
A_2(t,x,y)= \partial_{\eta}^2B_2(t,x,y,\partial_t
{\omega}(t,x),{\xi}_0(t)+ {\sigma}(t,x),0)/2 
\end{equation}
Here ${\sigma}(t,x)$ is given by~\eqref{dxomega} and
$\partial_t {\omega}(t,x)$ by~\eqref{dtomega}, where the expansion
will be up to a sufficiently high order in~$x$. 
Observe that after the change of variables we have ${\sigma}(t,0) \equiv 0$. 
The values of the symbols will as before be defined by finite Taylor
expansions in the ${\tau}$ and ${\xi}$ variables, which gives 
expansions in powers of~$x$ and $f(t,0,{\xi}(t))$.

We are going to construct solutions $ {\phi}_{k}(t,x,y) =
{\varphi}_k(t,x,{\varrho}y)$ so that
$y \mapsto {\varphi}_{k}(t,x,y) \in C^\infty_0$ uniformly in
${\varrho}$, which gives localization in $|y| \ls {\varrho}^{-1}$.  
Therefore we shall choose ${\varrho}y$ as new $y$ coordinates,
then~\eqref{transeq0} becomes
\begin{multline}
D_t {\phi} + A_0(t,x)D_x{\phi} +
{\varrho}A_1(t,x,y/{\varrho})D_y{\phi} +   
{\varrho}^2A_2(t,x,y/{\varrho})D_y^2{\phi} \\ + \sum_{j \ge 0}
{\varrho}^{-jN}  
R_j(t,y/{\varrho},x,D_t,D_x, {\varrho}D_{y}){\phi}
\end{multline}
By Proposition~\ref{eikprop} the phase function $e^{i{\lambda}w(t,x)}$
gives the cut-off in $x$, and we shall expand the symbols in powers
of~ $x$. Now the Taylor expansion of
$x \mapsto {\varrho}^2A_2(t,x,y/{\varrho})$ will give terms that are 
$\Cal O({\varrho}^2x)$. Therefore we take  ${\varrho}^2x$ as new
$x$ coordinates, which gives
\begin{multline}\label{trexp0}
D_t {\phi}  + {\varrho}^2 A_0(t,x/{\varrho}^2)D_x{\phi}  +
{\varrho}A_1(t,x/{\varrho}^2,y/{\varrho})D_y{\phi}  + 
{\varrho}^2A_2(t,x/{\varrho}^2,y/{\varrho})D_y^2{\phi}
\\ + \sum_{ j \ge 0} {\varrho}^{-jN}
R_j(t,y/{\varrho},x/{\varrho}^2,D_t,{\varrho}^2D_x, {\varrho}D_{y}){\phi} 
\end{multline}

Now the phase function 
$e^{i{\lambda}w(t,x)}  = \Cal O(e^{-c{\varrho}^{N-4}|x|^2})$
in the new coordinates. So if we take $N > 4$ it suffices to
solve the transport equation up to a sufficiently high order of $x$,
then we may cut off where $|x|\ls 1$, which corresponds to $|x| \ls
{\varrho}^{-2}$ in the original coordinates.
Thus we expand in $x$:
\begin{equation}\label{phiexp}
 {\phi}_k(t,x,y) =
 \sum_{k,{\alpha}}{\phi}_{k,{\alpha}}(t,y)x^{\alpha} \qquad
 {\phi}_{k,{\alpha}}(t,y) \in C^\infty_0
\end{equation}
$A_0(t,x/{\varrho}^2)D_x = 
\sum_{{\alpha},j}
A_{0,{\alpha},j}(t){\varrho}^{-2|{\alpha}|} x^{\alpha} D_{x_j}$, 
$$
A_j(t,x/{\varrho}^2,y/{\varrho}) = \sum _{\alpha} 
A_{j,{\alpha}}(t,y/{\varrho}){\varrho}^{-2|{\alpha}|} 
x^{\alpha}\qquad j > 0
$$
and
\begin{equation}\label{rkexp}
 R_k(t,x/{\varrho}^2,y/{\varrho},{\varrho}D_y, {\varrho}^2D_x) =
 \sum_{{\alpha},\ell,{\nu},{\mu}}
 R_{k,{\alpha},\ell,{\nu},{\mu}}(t,y/{\varrho})   
{\varrho}^{-2|{\alpha}|   +2|{\nu}| + |{\mu}|}
x^{\alpha}D_t^\ell D_x^{{\nu}}D_y^{\mu} 
\end{equation}
Here $\ell + |{\nu}| + |{\mu}| \le k + 2$ so we have at most the
factor ${\varrho}^{2|{\nu}|+ |{\mu}|} \le {\varrho}^{2k + 4}$
in~\eqref{rkexp}. When $k = 0$ we have $\ell + |{\nu}| +
|{\mu}| = 0$ and $R_0(t,x/{\varrho}^2,y/{\varrho}) = \sum _{\alpha}  
R_{0,{\alpha}}(t,y/{\varrho}){\varrho}^{-2|{\alpha}|} 
x^{\alpha}$. 
Observe that the coefficients in the expansions are given by
expansions in powers of $f(t,0,{\xi}(t))$. 
After cut-off in $x$ we find in the original coordinates that
${\phi}_k(t,x,y) = {\varphi}_{k}(t, 
{\varrho}^2x, {\varrho} y)$ where ${\varphi}_{k}$ for any $t$  
is uniformly bounded in $C^\infty_0$. 

We shall first apply~\eqref{trexp0} on $ \phi_0 $ and expand in~$x$. Then we find  
that the terms that are independent of ~$x$ are
\begin{multline}\label{rawtrans}
D_t {\phi}_{0,0}  -i {\varrho}^2 \sum_j A_{0,0,j}(t){\phi}_{0,e_j} +
{\varrho}A_{1,0}(t,y/{\varrho}) D_y{\phi}_{0,0} \\ + {\varrho}^2
A_{2,0}(t,y/{\varrho})D_y^2{\phi}_{0,0} +  R_{0,0}(t,y/{\varrho}){\phi}_{0,0} 
\end{multline} 
We shall need the following result, which gives
estimates on $f$ and~$A_j$
on the interval of integration. It will be proved in the next section. In the following, we shall
denote $f(t) = f(t,0,{\xi}_0(t))$ and $F(t) = \int_0 ^t f(s)\,ds$.
Observe that $f(0) = 0$ since $\im w_0'(0) = 0$.

\begin{lem}\label{intrem} 
Assume that the conclusions in Proposition~\ref{eikprop} hold and
that~\eqref{cond2a} holds if $t \mapsto f(t)$ vanishes of infinite 
order at $0$. Then there exists ${\varepsilon}$ and $C \ge 1$ with the property that if
$N \ge C$, $\varrho = {\lambda}^{1/N} \ge C$ and
\begin{equation}\label{intbound}
|f(t)| + \left| \int_0^t|A_{0}(s,0)| +
|A_{1}(s,0,y/{\varrho})| + \mn{A_{2}(s,0,y/{\varrho})}\,ds \right|
\ge C/{\varrho}^{3}  
\end{equation}
holds for some $|y| \le {\varrho}/C$, then $\lambda F(s) \le
- {\lambda}^{\varepsilon }/C$ for some $s$ in the interval
connecting $0$ and $t$.
\end{lem}

Observe that if Lemma~\ref{intrem} holds for some ${\varepsilon}$ and
$C$, then it trivially holds for smaller~${\varepsilon}$ and
larger~$C$.  We shall assume that $ \varepsilon < 1 $ and that both~ $N$
and~${\lambda}$ are large enough so that the conclusion in
Lemma~\ref{intrem} holds. Since~\eqref{intbound} does not hold when
$t=0$, we can choose the maximal interval ~$I$
containing $0$ such that~\eqref{intbound} does not hold in~$I$, thus 
\begin{equation}\label{intbound0}
|f(t)| + \left| \int_0^t|A_{0}(s,0)| +
|A_{1}(s,0,y/{\varrho})| + \mn{A_{2}(s,0,y/{\varrho})}\,ds \right|
< C/{\varrho}^{3}  \qquad t \in I 
\end{equation}
when $ |y| \le \varrho/C $.
By definition we obtain that~\eqref{intbound} holds for some 
$ |y| \le \varrho/C $ when  $t \in \partial I$  so Lemma~\ref{intrem} gives that ${\lambda}F \ls -
{\lambda}^{{\varepsilon}}$ at $\partial I_0$ for some open interval 
$I_0 \subseteq I $ that contains 0.
This means that  $e^{i{\lambda}{\omega}(t,0)}  =
e^{{\lambda}F(t)} \le  C_N{\lambda}^{-N}$ for any $N$ at $\partial I_0$ when
${\lambda} \gg 1$. Since $F' = f$ is uniformly bounded and 
the left hand side of~\eqref{intbound0} is Lipschitz continuous, 
we may cut off near $I_0 $ with ${\chi}(t) \in S(1,
{\lambda}^{6/N}dt^2) \subset S(1,
{\lambda}^{2-2{\varepsilon}}dt^2) $  for $N \gg 1$ so that $ \chi(0) \ne 0 $,
${\lambda}F(t) \ls -{\lambda}^{\varepsilon}$ in $ \supp \chi' $ and ~\eqref{intbound0} 
holds with some $ C $ when $t \in \supp \chi $ and $ |y| \le \varrho/C $.
Then as before the  cut-off errors can be
absorbed by the exponential and the expansion in powers of $ f(t,0,\xi(0)) = f(t) $ is justified.
In fact, $f(t) = \Cal O({\varrho}^{-3})$  in $ \supp \chi$,
which gives errors of any negative power of
${\varrho} = {\lambda}^{1/N}$. The bound on the integral
in~\eqref{intbound0} means that 
we can ignore the $A_j$ terms in~\eqref{rawtrans}  in $ \supp \chi$
modulo lower order terms in~${\varrho}$. In the following we shall change the notation and let $I = \supp \chi $. We need to measure the error terms in the following way.

\begin{defn}
For $ a(t) \in L^\infty(\br) $ and $ \kappa > 0 $ we say that $a(t) \in I({\kappa})$ if the integral $\int_0^t
a(s)\,ds = \Cal O ({\kappa})$ for all $t \in I$.
\end{defn} 

For example, $ f(t) \in   I(\varrho^{-3})$ and since the integral in~\eqref{intbound0} is 
$ \Cal O (\varrho^{-3})$ in $ I $ the integrand is in $  I(\varrho^{-3}) $.
Then according to~\eqref{intbound0} it suffices to solve   
\begin{equation}\label{6.12}
 D_t{\phi}_{0,0} =  -R_{0,0} {\phi}_{0,0}  \qquad t \in I 
\end{equation}
to obtain that the terms in~\eqref{rawtrans} are
in~$I({\varrho}^{-1})$, here $R_{0,0}(t,y/{\varrho}) \in C^\infty$ uniformly 
since ${\varrho} \ge 1$. Now we can solve~\eqref{6.12} with
${\phi}_{0,0}(0,y) = {\phi}(y) \in C_0^{\infty}$ uniformly  
with support where $ |y| \ll 1 $ such that ${\phi}(0) = 1$. In fact, the solution is ${\phi}_{0,0}(t,y)
= E(t,y){\phi}(y)$, where 
$$
E(t,y) = \exp\left(-i\int^t_0 R_{0,0}(s,y/{\varrho})\,ds\right)
\qquad t \in I   
$$
is uniformly bounded in $C^\infty$.
Thus ${\phi}_{0,0}(t,y) \in C^\infty$ uniformly and by choosing $ {\phi}(y) $
with sufficiently small support we obtain for any $t \in I$ 
that $  {\phi}_{0,0}(t,\cdot)$ has support in a
sufficiently small compact set in which~\eqref{intbound0} holds.

The coefficients of the terms in~\eqref{trexp0} which are
homogeneous in $x$ of degree~${\alpha} \ne 0$ in $x$ are
\begin{multline}\label{6.13}
D_t {\phi}_{0,{\alpha}} + R_{0,0}(t,y/{\varrho}){\phi}_{0,{\alpha}} -i
\sum_{\substack{|\beta| = 1\\j}} 
A_{0,{\beta},j}(t)({\alpha}_j + 1- {\beta}_j){\phi}_{0,{\alpha}+e_j
  -{\beta}} \\
+ \sum_{|\beta| = 1}A_{2,{\beta}}(t,y/{\varrho})
D_y^2{\phi}_{0,{\alpha}-{\beta}} 
\end{multline}
modulo~$I({\varrho}^{-1})$. Letting ${\Phi}_{k,j} 
= \set{{\phi}_{k,{\alpha}}}_{|{\alpha}| = j}$ and ${\Phi}_{k} =
\set{{\Phi}_{k,j}}_{j}$  for $k,\,j \ge 0$, we find that~\eqref{6.13}
vanishes if $ {\Phi}_{0}$ satisfies the system
\begin{equation}\label{0ktran}
 D_t {\Phi}_{0,k}   =   S_{0,0}^k{\Phi}_{0,k} +
  S_{0,1}^k{\Phi}_{0,k-1} 
\end{equation}
where $S_{0,0}^k(t)$ is a uniformly bounded matrix
depending on $t$, and $S_{0,1}^k(t,y/{\varrho},D_y)$ is a system of uniformly
bounded differential operators of order 2 in $y$ when $|y| \ls
{\varrho}$.
Let $E_{0,k}(t)$ be the fundamental solution to $D_t E_{0,k} =  S_{0,0}^k
E_{0,k}$ so that $E_{0,k}(0) = \id$, then letting ${\Phi}_{0,k}(t,y) =
E_{0,k}(t){\Psi}_{0,k}(t,y)$ the system~\eqref{0ktran} reduces to 
\begin{equation*}
 D_t {\Psi}_{0,k}(t,y) = E_{0,k}^{-1}S_{0,1}^k E_{0,k} {\Psi}_{0,k-1}(t,y)
\end{equation*}
This is a recursion equation which we can solve uniformly in~$I$ with
${\Psi}_{0,k}(t,y)$ having initial values
${\Psi}_{0,k}(0,y) \equiv 0$ for $0 < k \le M$. Observe that since the initial
data~${\Phi}_{0,k}(0,y)$ has  
compact support, we find that ${\Phi}_{0,k}(t,y)\in C^\infty$ uniformly.  For any~$ t $ we find that ${\Phi}_{0,k}(t,y)$ has support in a
sufficiently small compact set so that~\eqref{intbound0} holds for any $t \in I$.

We shall now apply~\eqref{trexp0} to $ \phi  $ given by the full expansion~\eqref{phiexp}. 
We find that the coefficients of the terms  in~\eqref{trexp0} which are homogeneous
of degree~${\alpha} \ne 0$ in~$x$ are equal to
\begin{multline} \label{6.15}
{\varrho}^{-1}\big(D_t {\phi}_{1,{\alpha}}   +  
R_{0,0}(t,y/{\varrho}){\phi}_{1,{\alpha}} - i
\sum_{\substack{|\beta| = 1\\j}} A_{0,{\beta},j}(t)({\alpha}_j + 1-
{\beta}_j){\phi}_{1,{\alpha}+e_j 
  -{\beta}} \\ 
+ \sum_{|\beta| = 1}A_{2,{\beta}}(t,y/{\varrho})
D_y^2{\phi}_{1,{\alpha}-{\beta}}  + \sum_{|\beta| =
  1}A_{1,{\beta}}(t,y/{\varrho}) D_y^2{\phi}_{0,{\alpha}-{\beta}} 
\\
 - i{\varrho}^3 \sum_j A_{0,0,j}(t)({\alpha}_j+1)
 {\phi}_{0,{\alpha}+e_j} 
+ {\varrho}^3 A_{2,0}(t,y/{\varrho})D_y^2{\phi}_{0,{\alpha}}  \big)
\end{multline}
modulo $I({\varrho}^{-2})$. We find that~\eqref{6.15} vanishes
if~${\Phi}_1$ satisfies the system 
\begin{equation}\label{tr1}
 D_t {\Phi}_{1,k} =  S_{1,0}^k{\Phi}_{1,k} +
 S_{1,1}^k{\Phi}_{1,k-1}  +  \mathbb A_1^0 {\Phi}_{0}
\end{equation}
where $S_{1,0}^k(t)$ is a uniformly bounded matrix
depending on $t$, $S_{1,1}^k(t,y/{\varrho},D_y)$ is a system of uniformly
bounded differential operators of order 2 when $|y| \ls
{\varrho}$ and $\mathbb A_1^0$ is a differential operator in~$y$ of
order~$2$ with coefficients in~$I(1)$ because of~\eqref{intbound0}. 
By letting ${\Phi}_{1,k} = E_{1,k} {\Psi}_{1,k}$ with the fundamental
solution $E_{1,k}$ to $D_tE_{1,k} 
=  S_{1,0}^kE_{1,k}$, $E_{1,k}(0) =  \id  $, this reduces to the equation 
$$
D_t {\Psi}_{1,k} =  E_{1,k}^{-1}
S_{1,1}^k E_{1,k-1}{\Psi}_{1,k-1}  + 
E_{1,k}^{-1}\mathbb A_1^0 {\Phi}_{0}
$$ 
Thus we can solve \eqref{tr1} in~$I$ recursively with
uniformly bounded ${\Phi}_{1,k}$ having initial values
${\Phi}_{1,k}(0,y) \equiv 0$, $k \ge 0$. But observe that $ {\Phi}_1  $ is not in
$ C^\infty $ uniformly, instead we have $D_t^j  
{\Phi}_1 = \Cal O({\varrho}^3)$ if $j \ge 1$, since
$|\partial_t^j \mathbb A_1^0| \le C_j{\varrho}^3$, $\forall\,
j$ by~\eqref{6.15}. For that reason, we shall define $ S^3_{\varrho} \subset C^\infty $ by
 \begin{equation}
 |\partial_t^j \partial_y^\alpha{\phi}(t,y)| \le C_{j,\alpha} {\varrho}^{3j} \quad \forall\, j \,\alpha
 \end{equation}
when ${\phi} \in S^3_{\varrho}$.
Observe that ${\phi} \in S^3_{\varrho}$ if and only if ${\phi}(t,y) =
{\chi}({\varrho}^3t,y)$ where ${\chi} \in C^\infty$ uniformly, and that the
operator ${\varrho}^{-3}D_t$ maps~$ S^3_{\varrho}  \mapsto
S^3_{\varrho}$.
Note that the expansion of the symbols also contains terms with factors
${\varrho}^3f^k$, $k \ge 1$, which are uniformly bounded in
$S^3_{\varrho}$ for~$t \in I$ by~\eqref{intbound0}.
Since $\int^t_0 \mathbb A_1^0\,dt \in S^3_{{\varrho}}$ in~$I$ we find
that ${\Phi}_1 \in S^3_{\varrho}$ in~$I$.

Recursively, the coefficients of the terms in~\eqref{trexp0}
that are homogeneous in $x$ of degree~
${\alpha}$ are
\begin{multline}\label{trexp}
{\varrho}^{-k} \big(D_t {\phi}_{k,{\alpha}} 
-i \sum_{{\beta}\ne 0} A_{0,{\beta},j}(t)({\alpha}_j + 1-{\beta}_j){\phi}_{k + 
  2 - 2|{\beta}|,{\alpha} + e_j - {\beta}} 
\\  
+ \sum_{{\beta}\ne 0} A_{1,{\beta}}(t,y/{\varrho})D_y{\phi}_{k + 1 -
2|{\beta}|,{\alpha} - 
{\beta}}
+ \sum_{{\beta}\ne 0} A_{2,{\beta}}(t,y/{\varrho})D_y^2{\phi}_{k
  + 2 - 2|{\beta}|,{\alpha} - 
  {\beta}} 
\\-i {\varrho}^3\sum_j
A_{0,0,j}(t)({\alpha}_j + 1){\phi}_{k-1,{\alpha} + e_j} +
{\varrho}^3 A_{1,0}(t,y/{\varrho})D_y{\phi}_{k-2,{\alpha}}
\\  + {\varrho}^3 A_{2,0}(t,y/{\varrho})D_y^2{\phi}_{k-1,{\alpha}} 
\\ + \sum_{\ell +  |{\nu}| +  |{\mu}| \le j + 2}
 {\varrho}^{-jN} R_{j,{\beta},\ell,{\nu},{\mu}}(t,y/{\varrho})
 c_{{\alpha},{\beta},{\nu}} 
{\varrho}^{-2|{\beta}|   +2|{\nu}| + |{\mu}| +i +3\ell}({\varrho}^{-3}
D_t)^\ell D_{y}^{\mu}{\phi}_{k-i, {\alpha} + {\nu} - {\beta}} \big)
\end{multline}
modulo $I({\varrho}^{-k-1})$. Here the last sum has 
 $\ell +  |{\nu}| +  |{\mu}| = 0$ when  $j = 0$, $({\varrho}^{-3} 
D_t)^\ell D_y^{\mu}$ maps $S^3_{\varrho} \mapsto S^3_{\varrho}$ and
the values of the symbols are given by a finite expansion in powers of~$f(t)$.

Since ${\phi}_j \in S^3_{\varrho}$
we obtain that the terms in~\eqref{trexp} are in~$I({\varrho}^{-k-1})$
if 
\begin{multline}\label{trexpk}
D_t {\phi}_{k,{\alpha}} 
-i \sum_{{\beta}\ne 0}
A_{0,{\beta},j}(t)({\alpha}_j + 1-{\beta}_j){\phi}_{k + 
  2 - 2|{\beta}|,{\alpha} + e_j - {\beta}} 
\\ 
 + \sum_{{\beta}\ne 0} A_{1,{\beta}}(t,y/{\varrho})D_y{\phi}_{k + 1 -
2|{\beta}|,{\alpha} - 
{\beta}}
+ \sum_{{\beta}\ne 0} A_{2,{\beta}}(t,y/{\varrho})D_y^2{\phi}_{k
  + 2 - 2|{\beta}|,{\alpha} - 
  {\beta}} 
\\ -i {\varrho}^3\sum_j
A_{0,0,j}(t)({\alpha}_j + 1){\phi}_{k-1,{\alpha} + e_j} +
{\varrho}^3 A_{1,0}(t,y/{\varrho})D_y{\phi}_{k-2,{\alpha}}
\\  + {\varrho}^3 A_{2,0}(t,y/{\varrho})D_y^2{\phi}_{k-1,{\alpha}} 
\\ = - \sum_{\substack {i + 3\ell + 2|{\nu}|+ |{\mu}|  = jN + 2|{\beta}|
    \\ \ell +  |{\nu}| +  |{\mu}| \le j + 2}}
 R_{j,{\beta},\ell,{\nu},{\mu}}(t,y/{\varrho})
 c_{{\alpha},{\beta},{\nu}}({\varrho}^{-3}
D_t)^\ell D_{y}^{\mu} {\phi}_{k-i, {\alpha} + {\nu} - {\beta}}
\end{multline}
When $j = 0$ we find that
$\ell +  |{\nu}| +  |{\mu}| = 0$, $ i =2|\beta| $ and we only have an expansion 
in~${\beta}$  in the last sum.
Now if $ j > 0 $, $\ell +  |{\nu}| +  |{\mu}| \le j + 2$ and
$
i + 3\ell + 2|{\nu}|+ |{\mu}| = jN + 2|{\beta}|
$ 
then we find that 
$$
jN \le i + 3\ell + 2|{\nu}| + |{\mu}| < i + 3(j+2)
$$ 
which gives $ i  \ge j(N-3) - 6 \ge N - 9 \ge 1$
if  $N \ge 10$. 
Thus we find that~\eqref{trexpk} can be written as
\begin{equation}\label{transeq}
 D_t {\Phi}_k = \mathbb A_0^k{\Phi}_k + \mathbb A_1^k {\Phi}_{k-1}
 + \mathbb A_{2}^k {\Phi}_{k-2} +  \dots
\end{equation}
where $\int_0^t \mathbb A_{j}^k \,dt$ is a uniformly bounded
differential operator on $S^3_{\varrho}$ for~$t \in I$ and $j > 0$. We
have that 
$$
\set{\mathbb A_0^k{\Phi}_k}_{j} =
S_{k,0}^j{\Phi}_{k,j} + S_{k,1}^j{\Phi}_{k,j-1} 
$$ 
where $S_{k,0}^j(t)$ is a uniformly bounded matrix
depending on $t$, and $S_{k,1}^j(t,y/{\varrho},D_y)$ is a system of
uniformly bounded differential operators of order~$2$ when $|y| \ls
{\varrho}$. 
By letting ${\Phi}_{k,j} = E_{k,j} {\Psi}_{k,j}$  with the fundamental 
solution $E_{k,j}  $ to $D_t E_{k,j} = S_{k,0}^j E_{k,j}$, $ E_{k,j}(0) = \id $, 
\eqref{transeq} becomes
a system of recursion equations in~$j$ and~$k$.
Thus~\eqref{transeq} can be solved in~$I$ with ${\Phi}_k \in S^3_{\varrho}$ having
initial values ${\Phi}_k(0) \equiv 0$, $k > 0$. We find from
\eqref{phiexp} and the definition of   
$S^3_{\varrho}$ that 
$
{\phi}_k(t,x,y) =  {\varphi}_k({\varrho}^3t,
{\varrho}^2x, {\varrho}y)
$
where ${\varphi}_k \in C^\infty$ uniformly when $ t \in I $. 
Thus we can solve the transport equation~\eqref{transexp} up to any
negative power of~${\lambda}$. Observe that by cutting off
in ~$t$ and~$x$ we may assume that ${\varphi}_k \in 
C_0^\infty$ has fixed compact support in $(x,y)$ and
support where $|t| \ls {\varrho}^3$.
It follows that the support of ${\phi}_k$ can be chosen in an
arbitrarily small 
neighborhood of~ ${\Gamma}$ for large enough~${\lambda}$.
Changing to the original coordinates, we obtain the
following result. 

\begin{prop}\label{transprop}
Assume that the conclusions in Proposition~\ref{eikprop} hold, and
that~\eqref{cond2a} is satisfied near~${\Gamma}$ when the sign change
 of $t \mapsto f(t,x_0,{\xi}_0)$ is of infinite order.
If ${\varrho} = {\lambda}^{1/N}$ for sufficiently large~$N$, then for
any $K$ and $M$ we can solve the transport 
equations~\eqref{trexpk} for $k \le K$ and $|{\alpha}| \le M$
near $\set{(t,x_0(t),y_0):\ t \in [t_1, t_2]}$.
By~\eqref{phiexp} this gives
$$
{\phi}_k(t,x,y) = {\varphi}_k({\varrho}^3(t-t_0),
{\varrho}^2(x-x_0(t)), {\varrho}(y-y_0))
\qquad k \le K
$$
where ${\varphi}_k(t,x,y)\in C^\infty$ uniformly, has support
where $ |x| + |y| \ls 1 $ and $ |t|\ls \varrho^3 $, and 
${\varphi}_0(0,0,0) = 1$ for some $t_0 
\in (t_1,t_2)$ such that\/ $\im w_0(t_0) = 0$.
\end{prop}

\section{The Rate of Change of Sign}

We have showed that $t \mapsto  f(t,x,{\xi})$ changes sign from $+$ to
$-$ on an interval $I$. Then   
\begin{equation}
 F(t) = \int^t f(s,x_0(s),{\xi}_0(s)) \,ds = \int^t f(s)\,ds
\end{equation}
has a local maximum in the interval. By choosing that maximum as
the starting point we may assume it is equal to $0$ so that $F(t) \le 0$. By
changing $t$~ coordinate, we may assume $F(0) = 0$. We shall study how the
size of the derivative $f$ affects the size of the function $F$.

\begin{lem}\label{intlem}
Assume that $0 \ge F(t) \in C^\infty$ has local maximum at $t = 0$,
and let $I_{t_0}$ be the closed interval
joining $0$ and $t_0\in \br$. If 
$$\max_{I_{t_0}}|F'(t)| = |F'(t_0)| = {\kappa}\le 1$$ 
with $|t_0| \ge {\kappa}^{\varrho}$ for some ${\varrho} > 0$,
then we have $\min_{I_{t_0}}  F(t) \le -
C_{\varrho}{\kappa}^{1+{\varrho}}$.
The constant $C_{\varrho}> 0$ only depends on
${\varrho}$ and the bounds on $F$ in $C^\infty$.
\end{lem}

\begin{proof} 
Let $f = F'$ then since $F(t) = F(0) + \int_0^t f(s)\,ds \le \int_0^t
f(s)\,ds$ it is no 
restriction to assume the maximum $F(0)=0$. By switching $t$ to $-t$
we may assume $t_0 \le -{\kappa}^{\varrho} < 0$. Let
\begin{equation}
 g(t) = {\kappa}^{-1}f(t_0 + t{\kappa}^{\varrho}) 
\end{equation}
then $|g(0)| = 1$, $|g(t)| \le 1$ for $0 \le t \le 1$ and 
$$|g^{(N)}(t)| = {\kappa}^{{\varrho}N
  -1}|f^{(N)} (t_0 + t{\kappa}^{\varrho})| \le C_N$$
when $N \ge 1/{\varrho}$  for $0 \le t \le 1$. By using the Taylor
expansion at $t= 0$ for $N \ge 1/{\varrho}$ we find 
\begin{equation}
 g(t) = p(t) + r(t)
\end{equation}
where $p$ is the Taylor polynomial of order $N-1$ of $g$ at $0$, and 
\begin{equation}
 r(t) = t^N \int_0^1g^{(N)}(ts)(1-s)^{N-1}\, ds/(N-1)! 
\end{equation}
is uniformly bounded in $C^\infty$ for $0 \le t \le 1$ and $ r(0) = 0 $. Since $g$
also is bounded on 
the interval, we find that $p(t)$ is uniformly bounded in $0 \le t \le 1$. Since
all norms on the finite dimensional space of polynomials of fixed
degree are equivalent, we find that $p^{(k)}(0) = g^{(k)}(0)$ are
uniformly bounded for $0 \le 
k < N$ which implies that $g(t)$ is uniformly bounded in $C^\infty$ for $0 \le t \le
1$. Since $|g(0)|= 1$ it exists a uniformly bounded ${\delta}^{-1} \ge
1$ such that  $|g(t)| \ge 1/2$
when $0 \le t \le {\delta}$, thus $g$ has the same sign in that
interval. Since $ g(s) = {\kappa}^{-1}f(t_0 +
s{\kappa}^{\varrho})$  we find
\begin{equation}
{\delta}/2 \le  \left|\int_0^{\delta} g(s)\,ds \right| = \left|
{\kappa}^{-\varrho}\int_{t_0}^{t_0 +
  \delta{\kappa}^{{\varrho}}} {\kappa}^{-1}f(t)\,dt \right| 
\end{equation}
Since $t_0 + {\delta}{\kappa}^{\varrho} \le 0$ we find that the variation of
$F(t)$ on $[t_0,0]$ is greater than ${\delta}{\kappa}^{1+{\varrho}}/2$ 
and since $F \le 0$ we find that the minimum of
$F$ on $I_{t_0}$ is smaller than  $-{\delta}{\kappa}^{1+{\varrho}}/2$.
\end{proof}

\begin{proof}[Proof of Lemma~\ref{intrem}] 
As before we let $F(t)$ satisfy $F(0) = 0 $ and $F'(t) = f(t)$ where 
$f(t) = f(t,0,{\xi}_0(t))$ satisfies $f(0) = 0$. We have assumed that
the estimate~\eqref{cond2a} holds near~$ \Gamma $ if $f(t)$
vanishes of infinite order at $t=0$.
Observe that the term $x_0'(t)$ in $A_0$ can be estimated by
$|\partial_w f(t,0,{\xi}_0(t))|$ by~\eqref{2}, which gives that
$|A_0(t,0)| \ls |\partial_w f(t,0,{\xi}_0(t))|$. We find from~\eqref{dtomega},
\eqref{2a} and~\eqref{2} that 
$ |\partial_t{\omega}(t,0)| \ls |f(t)| + |\partial_{w}f(t,0,{\xi}_0(t))| 
$ thus~\eqref{intbound} follows if
\begin{equation}\label{intbounda}
| f(t)| + \left| \int_0^t  | f(s)| + A_{0}(s,0) +
A_{1}(s,0,y/{\varrho}) + A_{2}(s,0,y/{\varrho}) \,ds \right|
\gs {\varrho}^{-3}  
\end{equation}
where now $A_0(t) = |\partial_{w}f(t,0,{\xi}_0(t))|$,
\begin{equation}
A_1(t,y/{\varrho}) =  |A(t,0,y/{\varrho},0, {\xi}_0(t),0)|
\end{equation}
and 
\begin{equation}
A_2(t,y/{\varrho}) = \mn{\partial_{{\eta}}^2
B(t,0,y/{\varrho},0,{\xi}_0(t),0)} 
\end{equation}
In the following we shall suppress the $y$ variables
in~\eqref{intbounda}, the results will be uniform when 
$|y| \le c{\varrho}$ for some~$ c > 0 $ since ~\eqref{cond2a} holds near ${\Gamma}$.
Observe that if $|f(s)|$ and
$|A_j(s)|$ are $\ll {\varrho}^{-3}$ for $0 \le j \le 2$ when $s$ is
between $0$ and $t$, then \eqref{intbounda} does not hold.

We shall first consider the case when $|f(t)|  \cong |t|^m$
vanishes of finite order at $t=0$. Then the order must be odd so we
find $F(t) = \int_0^t f(s)\,ds \le 0$ and $c \le |F(t)|/t^{2k} \le C <
0$ for some $k > 0$. Thus we find 
\begin{equation}
{\varrho}^{-3} \ls \left| \int_0^t  | f(s)| + A_0(s) + A_1(s) + A_2(s)\,ds\right| \ls |t| \ls |F(t)|^{1/2k}
\end{equation}
implies that $|F(t)| \gs {\varrho}^{-6k}$. Since ${\lambda} =
{\varrho}^N$ we then obtain ${\lambda}F(t) \ls -{\varrho}^{N-6k} \le
-{\varrho} = {\lambda}^{1/N}$ if $N > 6k$. The case when
$|t|^{2k-1} \cong |f(t)| \gs
{\varrho}^{-3}$ gives that $|t| \gs {\varrho}^{-3/2k-1}$ so 
${\lambda}F(t) \ls -{\varrho}^{N- \frac{6k}{2k-1}} \le -{\varrho}$ if $N > 6$.
Now one of these cases must hold if~\eqref{intbound} holds, so we get
the result in the finite vanishing case.

Next, we consider the infinite vanishing case, then we have assumed
that condition~\eqref{cond2a} holds, which means that 
\begin{equation*}
 \sum_{j=0}^{2}A_j(t) \ls |f(t)|^{\varepsilon} 
\end{equation*}
which implies that $A_j(0) =0$ for all $j$. Now we assume that
~\eqref{intbounda} holds at $t$, by switching $t$ and $-t$ we may
assume $t > 0$.
Then we obtain for  some $s \in
[0, t]$ that $|f(s)| \ge c{\varrho}^{-3}$ or $
A_j(s) \ge c{\varrho}^{-3}$ for some $c > 0$ and $j$. 
Now we define $t_0$ as the smallest $t_0 > 0$  such
that $|f(t_0)| =  c{\varrho}^{-3} $ or $A_j(t_0) =
c{\varrho}^{-3}$ for some $j$, then $t_0 \le t$. Then we obtain from
condition~\eqref{cond2a} in the 
first case that $c{\varrho}^{-3} = |f(t_0)| \ls
|f(t_0)|^{{\varepsilon}} $ and in the second case that 
\begin{equation}
 c{\varrho}^{-3} = A_j(t_0) \ls |f(t_0)|^{{\varepsilon}} 
\end{equation}
Since ${\varrho} = {\lambda}^{1/N}$ we find in both cases that
\begin{equation}\label{33}
 {\lambda}^{-3/{\varepsilon}N} = {\kappa} \le c|f(t_0)| \qquad c > 0
\end{equation}
where ${\lambda} \gg 1$ if and only if ${\kappa} \ll 1$. By taking the 
smallest $ t_0 $ such that~\eqref{33} is satisfied, we find that $ |f(t)| \le   |f(t_0)| $ 
for $ 0 \le t \le t_0 $.
Since $f(t)$ vanishes of infinite order at $t=0$, we find by using
Taylor's formula 
that $ |f(t)| \le C_M |t|^M$ for any positive integer
$M$. (Actually, it suffices to take $ M=1 $.)
Condition~\eqref{33} then gives
\begin{equation}
{\kappa}^{1/M} \ls |f(t_0)|^{1/M} \ls |t_0|
\end{equation}
so by using Lemma~\ref{intlem} with ${\varrho} = 1/M$ we find that 
\begin{equation}
 \min_{0 \le s \le t_0} F(s) \ls - {\kappa}^{1+ 1/M} =
 - {\lambda}^{-3(1+ 1/M)/{\varepsilon}N} \qquad {\lambda} \gg 1
\end{equation}
Thus we find that $\min_{0 \le s \le t_0} F(s) \ls -{\lambda}^{c-1}$
for some  $c > 0$ if $3(1+ 1/M)/{\varepsilon}N < 1$, i.e.,  $N > 3(1 +
1/M)/{\varepsilon}$, which gives Lemma~\ref{intrem}.
\end{proof}

\section{The proof of Theorem 2.7}\label{pfsect} 

We shall use the following modification
of Lemma 26.4.15 in~\cite{ho:yellow}. Recall that $\mn{u}_{(k)}$ is
the $L^2$ Sobolev norm of order $k$ of $u \in C_0^\infty$ and let
$\Cal D'_{{\Gamma}} = \set{u \in \Cal D': \wf (u) \subset {\Gamma}}$ for
$ \Gamma \subseteq T^*\br^n $.

\begin{lem}\label{estlem}
Let 
\begin{equation}\label{estlem0}
 u_{\lambda}(x) =  {\lambda}^{(n-1){\delta}/2}\exp(i{\lambda}{\omega}(x))
 \sum_{j=0}^M 
 {\varphi}_j ({\lambda}^{\delta}x){\lambda}^{-j{\varrho}} \qquad {\lambda} \ge 1
\end{equation}
with ${\varrho} > 0$, $0 < {\delta} < 1$,
${\omega} \in C^\infty (\br^n)$ satisfying $\im 
{\omega}\ge 0$, $|d \re{\omega}| \ge c > 0$, and
${\varphi}_j \in C^\infty_0(\br^n)$.
Here ${\omega}$ and 
${\varphi}_j$ may depend on ${\lambda}$ but uniformly, and
${\varphi}_j$ has fixed  
compact support in all but one of the 
variables, for which the support is bounded by $C{\lambda}^{\delta}$.
Then for any integer $N$ we have
\begin{equation}\label{estlem1}
 \mn{u_{\lambda}}_{(-N)} \le C {\lambda}^{-N} \qquad {\lambda} \ge 1
\end{equation}
If ${\varphi}_0(x_0) \ne 0$ and $\im {\omega}(x_0) = 0$ for some $x_0$ then
there exists $c > 0$ so that
\begin{equation}\label{estlem2}
  \mn{u_{\lambda}}_{(-N)} \ge c
  {\lambda}^{-N-\frac{n}{2} + \frac{(n-1){\delta}}2} \qquad {\lambda} \ge 1
  \qquad \forall\, N
\end{equation}
Let ${\Sigma} = \bigcap_{{\lambda} \ge 1} \bigcup_j  \supp
{\varphi}_j({\lambda}^{\delta}\cdot)$ and let $ {\Gamma}$ be the cone
generated by 
\begin{equation}\label{estlem3}
 \set{(x,\partial{\omega}(x)),\ x
   \in {\Sigma},\ \im {\omega}(x) = 0} 
\end{equation}
Then for any $k$ we find ${\lambda}^ku_{\lambda} \to 0$ in $\Cal
D'_{{\Gamma}}$ so ${\lambda}^k Au_{\lambda} \to 0$ in $C^\infty$ if
$A$ is a pseudodifferential operator such that $\wf(A) \cap {\Gamma} =
\emptyset$. The estimates are uniform if
${\omega} \in C^\infty$ uniformly with fixed lower bound on
$|d\re{\omega}|$, and 
${\varphi}_j \in C^\infty_0$ uniformly with the support condition.
\end{lem}

In the expansion~\eqref{estlem0} we shall take ${\varrho} = 1/N$ and
${\delta} =  3/N$ with $N > 3$, and the cone $ {\Gamma}$ will be generated by  
\begin{equation}
 \set{(t,x_0(t),y_0,0,{\xi}_0(t),0): \ t \in I}
\end{equation}
where $I = \set{t: \im w_0(t) = 0}$.
Observe that the phase function in~\eqref{omegaexp} will satisfy the
conditions in Lemma~\ref{estlem} near
$\set{(t,x_0(t), y_0): \ t \in I}$  
since ${\xi}_0(t) \ne 0$ and $\im {\omega}(t,x) \ge 0$
by Proposition~\ref{eikprop}. Also, we find from
Proposition~\ref{transprop} that the functions ${\varphi}_k$ will
satisfy the conditions in Lemma~\ref{estlem} with~${\delta}= 3/N$
after making the change of
variables $(t,x,y) \mapsto (t-t_0,x-x_0(t), y-y_0)$ since
${\varphi}_0(t_0,x_0(t_0),y_0) = 1$.  
Observe that the conclusions of Lemma~\ref{estlem} are invariant under
uniform changes of coordinates.

\begin{proof}[Proof of Lemma \ref{estlem}]
We shall modify the proof of~\cite[Lemma 26.4.15]{ho:yellow} to this case. 
We have that 
\begin{equation}
 \hat u_{\lambda}({\xi}) = {\lambda}^{(n-1){\delta}/2}\sum_{j=0}^{M}
 {\lambda}^{-j{\delta}} \int
 e^{i{\lambda}{\omega} (x)
 - i\w{x,{\xi}}} {\varphi}_j({\lambda}^{\delta}x)\,dx 
\end{equation}
Let $U$ be a neighborhood of the projection on the second component of the set
in~\eqref{estlem3}. When ${\xi}/{\lambda} \notin U $ then for
${\lambda} \gg 1 $ we find that
$$
\bigcup_j \supp {\varphi}_j({\lambda}^{\delta}\cdot) \ni  x \mapsto
({\lambda}{\omega}(x) - \w{x,{\xi}})/({\lambda} + |{\xi}|)
$$
is in 
a compact set of functions with non-negative imaginary part with a fixed
lower bound on the gradient of the real part. Thus, by integrating by
parts we find for any positive integer $k$ that   
\begin{equation}\label{pfest}
 |\hat u_{\lambda}({\xi})| \le C_k{\lambda}^{((n-1)/2+k){\delta}}({\lambda}
 +|{\xi}|)^{-k}\qquad 
 {\xi}/{\lambda} \notin U \qquad {\lambda} \gg 1
\end{equation}
which gives any negative power of ${\lambda}$ for $k$ large enough,
since ${\delta} <
1$. If $V$ is bounded and $0 \notin \ol V$ then since $u_{\lambda}$ is
uniformly bounded in $L^2$ we find
\begin{equation}
 \int_{{\lambda}V}  |\hat u_{\lambda}({\xi})|^2 (1 +
 |{\xi}|^2)^{-N}\,d{\xi} \le C_V{\lambda}^{-2N}
\end{equation}
which together with~\eqref{pfest} gives~\eqref{estlem1}. If ${\chi}
\in C_0^\infty$ then we may apply~\eqref{pfest} to
${\chi}u_{\lambda}$, thus we find for any
positive integer $k$ that
\begin{equation}
  |\widehat {{\chi}u}_{\lambda}({\xi})| \le
  C{\lambda}^{((n-1)/2+k){\delta}} ({\lambda}+
  |{\xi}|)^{-k} \qquad  {\xi} \in W \qquad {\lambda} \gg 1
\end{equation}
if $W$ is any closed cone with $(\supp
{\chi}\times W) \bigcap  {\Gamma}  = \emptyset$. Thus we find that
${\lambda}^ku_{\lambda} \to 0$ in $\Cal D'_{{\Gamma}}$ for every $k$.
To prove \eqref{estlem2} we may assume that $x_0 = 0$ and take ${\psi}\in
C_0^\infty$. If $\im {\omega}(0) = 0$ and ${\varphi}_0(0) \ne
0$ then since ${\delta} < 1$ we obtain that
\begin{multline}\label{limit}
 {\lambda}^{n -(n-1){\delta}/2} e^{- i \lambda \re w(0)}\w{u_{\lambda}, {\psi}({\lambda}\cdot)} = \int
 e^{i{\lambda}(w(x/{\lambda}) - \re w(0))}{\psi}(x)
 \sum_{j}{\varphi}_j({\lambda}^{{\delta}-1} x)
 {\lambda}^{-j{\delta}}\,dx \\ \to \int 
 e^{i\w{\re \partial_x{\omega}(0),x}}{\psi}(x)
 {\varphi}_0(0)\,dx \qquad {\lambda} \to \infty
\end{multline}
which is not equal to zero for some suitable ${\psi}  \in
C^\infty_0$. Since 
\begin{equation}
 \mn{{\psi}({\lambda}\cdot)}_{(N)} \le C_N {\lambda}^{N-n/2}
\end{equation}
we obtain from~\eqref{limit} that $0 < c \le  {\lambda}^{N + n/2 - (n-1){\delta}/2}
\mn{u_{\lambda}}_{(-N)}$ which gives~\eqref{estlem2} and the lemma. 
\end{proof}

\begin{proof}[Proof of Theorem~\ref{mainthm}] 
By conjugating with elliptic Fourier integral operators and
multiplying with pseudodifferential operators, we may obtain 
that $P^* \in {\Psi}^{2}_{cl}$ is on the form given by Proposition~\ref{prepprop}
microlocally near~${\Gamma} = \set{(t,x_0,y_0,0,{\xi}_0,0):\ t \in I}$. Thus we may assume
\begin{equation} 
P^* = D_t + F(t,x,y,D_t,D_x,D_y) + R
\end{equation}
where  $R\in
{\Psi}^2_{cl}$ satisfies $\wf (R) \bigcap
{\Gamma} = \emptyset$. 

Now we can construct approximate solutions $u_{\lambda}$ on the
form~\eqref{udef} by using the 
expansion~\eqref{exp}. By reducing to minimal bicharacteristics, we
may solve first the eikonal equation by using
Proposition~\ref{eikprop} and then the transport
equations~\eqref{trexpk} by using Proposition~\ref{transprop} with
${\varrho} = {\lambda}^{1/N}$ for $N > 3$. 
Thus after making the change of coordinates $(t,x,y)
\mapsto (t -t_0,x-x_0(t), y-y_0)$ we obtain approximate solutions
$u_{\lambda}$ on the  
form~\eqref{estlem0} in Lemma~\ref{estlem} with ${\varrho} = 1/N$ and
${\delta} = 3/N$. For $N$ large enough, we may choose $K$ and $M$ in
Proposition~\ref{transprop} so that 
$|(D_t + F)u_{\lambda}| \ls {\lambda}^{-k}$ for any $k$.
Now differentiation of $(D_t + F)u_{\lambda}$ can at most give a
factor ${\lambda}$ since ${\delta} < 1$ and a loss of 
a factor $x-x_0(t)$ gives at most a factor ${\lambda}^{1/2}$. 
Because of the bounds on the
support of $u_{\lambda}$ we may obtain that 
\begin{equation} \label{8.13}
\mn{(D_t + F){u_{\lambda}}}_{({\nu})} = \Cal
O({\lambda}^{-N-n})
\end{equation} 
for any chosen ${\nu}$. 
Since ${\varphi}_0(t_0,x_0(t_0),y_0) = 1$ by
Proposition~\ref{transprop} and $\im w(t_0,x_0(t_0)) =
0$ by Proposition~\ref{eikprop}, we find
by~\eqref{estlem1}--\eqref{estlem2} that
\begin{equation}\label{lastest}
{\lambda}^{-N - \frac{n}{2}} \ll
{\lambda}^{-N-\frac{n}{2} + \frac{(n-1){\delta}}2} \ls \mn{u}_{(-N)}  
\ls  {\lambda}^{-N} \qquad \forall\, N \qquad {\lambda} \gg 1
\end{equation}  
Since $u_{\lambda}$ has support in a fixed
compact set that shrinks towards $\set{(t,x_0(t),y_0):\ t \in I}$ as ${\lambda}
\to \infty$, 
we find from Lemma~\ref{estlem} that $\mn{Ru}_{({\nu})}$ and
$\mn{Au}_{(0)}$ are $\Cal O({\lambda}^{-N-n})$ if $\wf(A)$ does not
intersect ${\Gamma}$. Thus we find from~\eqref{8.13}
and~\eqref{lastest} that~\eqref{solvest} does 
not hold when ${\lambda} \to \infty$, so $P$ is not
solvable at~${\Gamma}$ by Remark~\ref{solvrem}. 
\end{proof}

\bibliographystyle{plain}

\bibliography{nec}

\begin{thebibliography}{10}

\bibitem{CT}
Cardoso, F. and F. Treves,
\newblock A necessary condition of local solvability for pseudo-differential
  equations with double characteristics,
\newblock {\em Ann. Inst. Fourier (Grenoble)}, 24(1), 225--292, 1974.

\bibitem{CCP}
Colombini, F., P. Cordaro, and L. Pernazza,
\newblock Local solvability for a class of evolution equations,
\newblock {\em J. Funct. Anal.}, 258(10), 3469--3491, 2010.

\bibitem{CPT}
Colombini, F., L. Pernazza and  F. Treves,
\newblock Solvability and nonsolvability of second-order evolution equations,
\newblock in {\em Hyperbolic problems and related topics}, Grad. Ser. Anal.,
  111--120, Int. Press, Somerville, MA, 2003.

\bibitem{de:nt}
Dencker, N.,
\newblock The resolution of the {N}irenberg-{T}reves conjecture,
\newblock {\em Ann. of Math. (2)}, 163(2), 405--444, 2006.

\bibitem{de:limit}
\bysame,
\newblock Solvability and limit bicharacteristics, 
\newblock  {\em J. Pseudo-Differ. Oper. Appl.} 7(3), 295--320, 2016.


\bibitem{Ego}
Egorov, Ju.~V.,
\newblock Solvability conditions for equations with double characteristics,
\newblock {\em Dokl. Akad. Nauk SSSR}, 234(2), 280--282, 1977.

\bibitem{GT}
{Gilioli, A. and F. Treves},
\newblock {An example in the solvability theory of linear {PDE}'s},
\newblock {\em Amer. J. Math.}, {96}, {367--385}, {1974}.

\bibitem{Gold}
{Goldman, R.,}
\newblock  {A necessary condition for the local solvability of a
              pseudodifferential equation having multiple characteristics},
\newblock   {\em J. Differential Equations}, {19(1)}, {176--200}, {1975}

\bibitem{ho:cauchy} 
H{\"o}rmander, L., {The {C}auchy problem for differential
  equations with double characteristics},
 {\em J. Analyse Math.}, {32}, {1977}, {118--196}.

\bibitem{ho:nec}
\bysame, 
\newblock Pseudodifferential operators of principal type,
\newblock in {\em Singularities in boundary value problems ({M}aratea,
  1980)}, {\em NATO Adv. Study  
  Inst. Ser. C: Math. Phys. Sci. 65}, 69--96. Reidel, Dordrecht, 1981.

\bibitem{ho:yellow} 
\bysame, 
\newblock {\em The Analysis of Linear Partial Differential Operators, Vol.  I--IV},
\newblock Springer Verlag, Berlin, Heidelberg, New York, Tokyo, 1983--1985.

\bibitem{MU1}
Mendoza, G.\ A.  and G.\ A. Uhlmann,
\newblock A necessary condition for local solvability for a class of operators
  with double characteristics,
\newblock {\em J. Funct. Anal.}, 52(2), 252--256, 1983.

\bibitem{MU2}
\bysame, 
\newblock A sufficient condition for local solvability for a class of operators
  with double characteristics,
\newblock {\em Amer. J. Math.}, 106(1), 187--217, 1984.

\bibitem{Men}
Mendoza, G.\ A.,
\newblock A necessary condition for solvability for a class of operators with
  involutive double characteristics,
\newblock in {\em Microlocal analysis ({B}oulder, {C}olo., 1983)}, 
  {\em Contemp. Math. 27}, 193--197, Amer. Math. Soc., Providence, RI, 1984.

\bibitem{Nishi}
Nishitani, T.,
\newblock Effectively hyperbolic {C}auchy problem,
\newblock in {\em Phase space analysis of partial differential equations
  {V}ol. {II}}, Pubbl. Cent. Ric. Mat. Ennio Giorgi, 363--449, Scuola
  Norm. Sup., Pisa, 2004.

\bibitem{Pop}
Popivanov, P.,
\newblock The local solvability of a certain class of pseudodifferential
  equations with double characteristics,
\newblock {\em C. R. Acad. Bulgare Sci.}, 27, 607--609, 1974.

\bibitem{Treves}
Treves, F.,
\newblock Concatenations of second-order evolution equations applied to local
  solvability and hypoellipticity,
\newblock {\em Comm. Pure Appl. Math.}, 26, 201--250, 1973.

\bibitem{T2}
\bysame, 
\newblock {\em Introduction to pseudodifferential and {F}ourier integral
  operators, {V}ol. 2},
\newblock The University Series in Mathematics, Plenum Press, New York--London, 1980.

\bibitem{Wen1}
Wenston, P.\ R.,
\newblock A necessary condition for the local solvability of the operator
  {$P_{m}^{2}(x,D)+P_{2m-1}(x,D)$},
\newblock {\em J. Differential Equations}, 25(1), 90--95, 1977.

\bibitem{Wen2}
\bysame, 
\newblock A local solvability result for operators with characteristics having
  odd order multiplicity,
\newblock {\em J. Differential Equations}, 28(3), 369--380, 1978.

\bibitem{Witt}
Wittsten, J.,
\newblock On some microlocal properties of the range of a pseudodifferential
  operator of principal type.
\newblock {\em Anal. PDE}, 5(3):423--474, 2012.

\bibitem{Yama1}
Yamasaki, A.,
\newblock On a necessary condition for the local solvability of
  pseudodifferential operators with double characteristics,
\newblock {\em Comm. Partial Differential Equations}, 5(3), 209--224, 1980.

\bibitem{Yama2}
\bysame, 
\newblock On the local solvability of {$D^{2}_{1}+A(x_{2},\,D_{2})$},
\newblock {\em Math. Japon.}, 28(4), 479--485, 1983.

\end{thebibliography}

\end{document}